\documentclass{article}
\usepackage[table]{xcolor}
\usepackage[utf8]{inputenc}
\usepackage{amsmath,amssymb,amsthm}
\usepackage{graphicx}
\usepackage[colorlinks=true,citecolor=black,linkcolor=black,urlcolor=blue]{hyperref}
\usepackage{url}
\usepackage{cleveref}
\usepackage{subcaption}
\usepackage{todonotes}
\usepackage{enumitem}
\usepackage{diagbox}

\newcommand{\arxiv}[1]{\href{http://arxiv.org/abs/#1}{\texttt{arXiv:#1}}}

\def\paren#1{\left( #1 \right)}
\def\acc#1{\left\{ #1 \right\}}

\renewcommand{\le}{\leqslant}
\renewcommand{\ge}{\geqslant}
\newcommand{\LL}{\mathcal L}

\theoremstyle{plain}
\newtheorem{theorem}{Theorem}
\newtheorem{lemma}[theorem]{Lemma}
\newtheorem{corollary}[theorem]{Corollary}
\newtheorem{proposition}[theorem]{Proposition}

\newtheorem{observation}[theorem]{Observation}

\theoremstyle{definition}

\newtheorem{example}[theorem]{Example}

\theoremstyle{remark}
\newtheorem{remark}[theorem]{Remark}

\title{Critical exponent of binary words with few distinct palindromes}
\author{L\!'ubom{\'i}ra Dvo\v{r}\'akov\'a\footnote{FNSPE Czech Technical University, Prague, Czech Republic.\\
e-mail \href{mailto:lubomira.dvorakova@fjfi.cvut.cz}{\tt lubomira.dvorakova@fjfi.cvut.cz}.}
\and
Pascal Ochem\footnote{LIRMM, CNRS, Universit\'e de Montpellier, France.\\
e-mail \href{mailto:ochem@lirmm.fr}{\tt ochem@lirmm.fr}.}
\and
Daniela Opo\v{c}ensk\'a\footnote{FNSPE Czech Technical University, Prague, Czech Republic.\\
e-mail \href{mailto:opocedan@fjfi.cvut.cz}{\tt opocedan@fjfi.cvut.cz}.}}
\date{2023}

\begin{document}

\maketitle

\begin{abstract}
We study infinite binary words that contain few distinct palindromes.
In particular, we classify such words according to their critical exponents.
This extends results by Fici and Zamboni [TCS 2013].
Interestingly, the words with 18 and 20 palindromes happen to be morphic images of the fixed point
of the morphism $\texttt{0}\mapsto\texttt{01}$, $\texttt{1}\mapsto\texttt{21}$, $\texttt{2}\mapsto\texttt{0}$.
\end{abstract}

\section{Introduction}\label{sec:intro}
We consider the trade-off between the number of distinct palindromes and the critical exponent in infinite binary words.
For brevity, every mention of a number of palindromes will refer to a number of distinct palindromes, including the empty word.
Fici and Zamboni~\cite{FZ2013} show that the least number of palindromes in an infinite binary word
is 9 and this bound is reached by the word $(\texttt{001011})^\omega$.
At the other end of the spectrum, the famous Thue-Morse word $TM$, fixed point of the morphism
$\texttt{0}\to\texttt{01}$, $\texttt{1}\to\texttt{10}$, has the least critical exponent
and infinitely many palindromes.

Our results completely answer questions of the form
``Is there an infinite $\beta^+$-free binary word with at most $p$ palindromes?''
In each case, we also determine whether there are exponentially or polynomially many such words. 
The results are summarized in~\Cref{tab1}.
A green (resp. red) cell means that there are exponentially (resp. polynomially) many words.
We have labelled the cells that correspond to an item of Theorem~\ref{thm:exp_pairs} or~\ref{thm:fc}.
Fici and Zamboni~\cite{FZ2013} also show that an aperiodic binary word contains at least 11 palindromes and this bound
is reached by the morphic image of the Fibonacci word by $\texttt{0}\to\texttt{0}$, $\texttt{1}\to\texttt{01101}$.
This word contains in particular the factor\\
{\scriptsize $$(\texttt{00110100011010011010001101000110100110100011010011010001101000110100110100011010001101})^{\tfrac72}.$$}
Theorem 1.(a) improves this exponent to $\tfrac{10}3^+$.
Fleischer and Shallit~\cite{FleischerShallit} have considered the number of binary words
of length $n$ with at most 11 palindromes (sequence \href{https://oeis.org/A330127}{A330127} in the OEIS) and proved that
it is $\Theta\paren{\kappa^n}$, where $\kappa=1.1127756842787\ldots$ is the root of $X^7=X+1$.

\begin{table}[!htb]
\centerline{
\begin{tabular}{|l|l|l|l|l|l|l|l|l|l|l|}
\hline
$\infty$ & $TM$\cellcolor{red!50} & \cellcolor{green} & \cellcolor{green}\cellcolor{green} & \cellcolor{green} & \cellcolor{green} & \cellcolor{green} & \cellcolor{green} & \cellcolor{green} & \cellcolor{green} & \cellcolor{green} \\
\hline
25 &  & \cellcolor{green}3.h & \cellcolor{green} & \cellcolor{green} & \cellcolor{green} & \cellcolor{green} & \cellcolor{green} & \cellcolor{green} & \cellcolor{green} & \cellcolor{green} \\
\hline
24 &  &  &  \cellcolor{green} & \cellcolor{green} & \cellcolor{green} & \cellcolor{green} & \cellcolor{green} & \cellcolor{green} & \cellcolor{green} & \cellcolor{green}\\
\hline
23 &  &  & \cellcolor{green} & \cellcolor{green} & \cellcolor{green} & \cellcolor{green} & \cellcolor{green} & \cellcolor{green} & \cellcolor{green} & \cellcolor{green}\\
\hline
22 &  &  & \cellcolor{green} & \cellcolor{green} & \cellcolor{green} & \cellcolor{green} & \cellcolor{green} & \cellcolor{green} & \cellcolor{green} & \cellcolor{green}\\
\hline
21 &  &  & 3.g\cellcolor{green} & \cellcolor{green} & \cellcolor{green} & \cellcolor{green} & \cellcolor{green} & \cellcolor{green} & \cellcolor{green} & \cellcolor{green}\\
\hline
20 &  &  & \cellcolor{red!50}7.b & \cellcolor{green} & \cellcolor{green} & \cellcolor{green} & \cellcolor{green} & \cellcolor{green} & \cellcolor{green} & \cellcolor{green}\\
\hline
19 &  &  &  & 3.f\cellcolor{green} & \cellcolor{green} & \cellcolor{green} & \cellcolor{green} & \cellcolor{green} & \cellcolor{green} & \cellcolor{green}\\
\hline
18 &  &  &  & \cellcolor{red!50}7.a & 3.e \cellcolor{green} & \cellcolor{green} & \cellcolor{green} & \cellcolor{green} & \cellcolor{green} & \cellcolor{green}\\
\hline
17 &  &  &  &  &  & \cellcolor{green} & \cellcolor{green} & \cellcolor{green} & \cellcolor{green} & \cellcolor{green}\\
\hline
16 &  &  &  &  &  & \cellcolor{green} & \cellcolor{green} & \cellcolor{green} & \cellcolor{green} & \cellcolor{green}\\
\hline
15 &  &  &  &  &  & 3.d\cellcolor{green} & \cellcolor{green} & \cellcolor{green} & \cellcolor{green} & \cellcolor{green}\\
\hline
14 &  &  &  &  &  &  & \cellcolor{green} & \cellcolor{green} & \cellcolor{green} & \cellcolor{green} \\
\hline
13 &  &  &  &  &  &  & 3.c\cellcolor{green} & \cellcolor{green} & \cellcolor{green} & \cellcolor{green} \\
\hline
12 &  &  &  &  &  &  &  & 3.b \cellcolor{green}& \cellcolor{green} & \cellcolor{green} \\
\hline
11 &  &  &  &  &  &  &  &  & 3.a \cellcolor{green}& \cellcolor{green} \\
\hline
10 &  &  &  &  &  &  &  &  &  & \cellcolor{red!50} \\
\hline
9 &  &  &  &  &  &  &  &  &  &\cellcolor{red!50}$(\texttt{001011})^\omega$ \\
\hline
\slashbox{$p$}{$\beta^+$} & $2^+$ & $\tfrac73^+$ & $\tfrac52^+$ & $\tfrac{28}{11}^+$ & $\tfrac{13}5^+$ & $\tfrac83^+$ & $3^+$ & $\tfrac{23}7^+$ & $\tfrac{10}3^+$ & $\infty$ \\
\hline
\end{tabular} }
\caption{Infinite $\beta^+$-free binary words with at most $p$ palindromes. }\label{tab1}
\end{table}

\section{Preliminaries}\label{sec:pre}
An \emph{alphabet} $\mathcal A$ is a finite set and its elements are called \emph{letters}. 
A \emph{word} $u$ over $\mathcal A$ of \emph{length} $n$ is a finite string $u = u_0 u_1 \cdots u_{n-1}$, where $u_j\in\mathcal A$ for all $j \in \{0,1,\dots, n-1\}$.
If ${\mathcal A}=\{\tt 0, \tt 1, \dots, \tt d-1\}$, the length of $u$ is denoted $|u|$ and $|u|_{\tt i}$ denotes the number of occurrences of the letter ${\tt i}\in\mathcal A$ in the word $u$. The \emph{Parikh vector} $ \vec{u} \in \mathbb N^{d}$ is the vector defined as ${\vec u } = (|u|_{\tt 0}, |u|_{\tt 1}, \dots, |u|_{\tt d-1})^{ T}$.
The set of all finite words over $\mathcal A$ is denoted ${\mathcal A}^*$. The set ${\mathcal A}^*$ equipped with concatenation as the operation forms a monoid with the \emph{empty word} $\varepsilon$ as the neutral element.
We will also consider the set ${\mathcal A}^\omega$ of infinite words (that is, right-infinite words) and the set ${}^\omega{\mathcal A}^\omega$ of bi-infinite words.
A word $v$ is an $e$-\emph{power} of a word $u$ if $v$ is a~prefix of the infinite periodic word $uuu\cdots = u^\omega$ and $e=|v|/|u|$.
We write $v=u^e$. We also call $u^e$ a repetition with period $u$ and exponent $e$.
For instance, the Czech word $kapka$ (drop) can be written in this formalism as $(kap)^{5/3}$.  
A word is \emph{$\alpha^+$-free} (resp. \emph{$\alpha$-free}) if it 
contains no repetition with exponent $\beta$ such that $\beta>\alpha$ (resp. $\beta\ge\alpha$).

The \emph{critical exponent}  $E({\bf u} )$ of an infinite word ${\bf u}$ is defined as
$$E({\bf u}) =\sup\{e \in \mathbb{Q}: \  u ^e \  \text{is a factor of   } {\bf u}  \  \text{for a non-empty word} \  u\}\,.
$$
The \emph{asymptotic critical exponent} $E^*(\bf u)$ of an infinite word ${\bf u}$ is defined as $+\infty$ if $E({\bf u}) = +\infty$, and
$$E^*({\bf u}) =\limsup_{n\to \infty}\{e \in \mathbb{Q}: \  u ^e \  \text{is a factor of  } {\bf u}  \  \text{for some }  u \ \text{of length} \  n  \}\,,$$
 otherwise.
If each factor of ${\bf u}$ has infinitely many occurrences in ${\bf u}$, then ${\bf u}$ is \emph{recurrent}.
Moreover, if for each factor the distances between its consecutive occurrences are bounded, then ${\bf u}$ is \emph{uniformly recurrent}. The \emph{language} $\mathcal{L}(\bf u)$ is the set of factors occurring in $\bf u$.
The language $\mathcal{L}(\bf u)$ is \emph{closed under reversal} if for each factor $w=w_0w_1\cdots w_{n-1}$, its \emph{reverse} $w^R=w_{n-1}\cdots w_1 w_0$ is also a~factor of $\bf u$.
A word $w$ is a~\emph{palindrome} if $w=w^R$. Let us denote $\overline{\tt 0}=\tt 1$ and $\overline{\tt 1}=\tt 0$, then for any binary word $w$ its \emph{bit complement} is $\overline{w}=\overline{w_0}\ \overline{w_1}\cdots \overline{w_{n-1}}$.

Consider a factor $w$ of a recurrent infinite word ${\bf u} = u_0 u_1 u_2 \cdots$. Let $j < \ell$ be two consecutive occurrences of $w$ in $\bf u$. Then the word $u_j u_{j+1} \cdots u_{\ell-1}$ is a~\emph{return word} to $w$ in $\bf u$.

The \emph{(factor) complexity} of an infinite word ${\bf u}$ is the mapping ${\mathcal C}_{\bf u}: \mathbb N \to \mathbb N$ defined by
${\mathcal C}_{\bf u}(n) = \# \{w \in \LL({\bf u}) : |w| =  n \}$.

Given a word $w \in \LL({\bf u})$, we define the sets of left extensions, right extensions and bi-extensions of $w$ in ${\bf u}$ over an alphabet $\mathcal A$ respectively as
$$
L_{{\bf u}}(w) = \{ {\tt i} \in {\mathcal A} : {\tt i}w \in \LL({\bf u}) \},
\qquad
R_{{\bf u}}(w) = \{ {\tt j} \in {\mathcal A} : w{\tt j} \in \LL({\bf u}) \}
$$
and
$$
B_{{\bf u}}(w) = \{ ({\tt i},{\tt j}) \in {\mathcal A} \times {\mathcal A} : {\tt i}w{\tt j} \in \LL({\bf u}) \}. $$

If $\#L_{{\bf u}}(w)>1$, then $w$ is called \emph{left special (LS)}. If $\#R_{{\bf u}}(w)>1$, then $w$ is called \emph{right special (RS)}. If $w$ is both LS and RS, then it is called \emph{bispecial (BS)}.
We define $b(w)=\#B_{{\bf u}}(w)-\#L_{{\bf u}}(w)-\#R_{{\bf u}}(w)+1$ and we distinguish \emph{ordinary BS factors} with $b(w)=0$, \emph{weak BS factors} with $b(w)<0$ and \emph{strong BS factors} with $b(w)>0$.

A \emph{morphism} is a map $\psi: {\mathcal A}^* \to {\mathcal B}^*$ such that $\psi(uv) = \psi(u)\psi(v)$  for all words $u, v \in {\mathcal A}^*$.
The morphism $\psi$ is \emph{non-erasing} if $\psi(\tt i)\not =\varepsilon$ for each ${\tt i} \in {\mathcal A}$.
Morphisms can be naturally extended to infinite words by setting
$\psi(u_0 u_1 u_2 \cdots) = \psi(u_0) \psi(u_1) \psi(u_2) \cdots\,$.
A \emph{fixed point} of a morphism $\psi:  {\mathcal A}^* \to  {\mathcal A}^*$ is an infinite word $\bf u$ such that $\psi(\bf u) = \bf u$.
We associate to a morphism $\psi: {\mathcal A}^* \to  {\mathcal A}^*$ the \emph{(incidence) matrix} $M_\psi$ defined for each $k,j \in \{0,1,\dots, d-1\}$ as $[M_\psi]_{kj}=|\psi(\tt j)|_{\tt k}$. 

If there exists $N\in \mathbb N$ such that $M_\psi^N$ has positive entries, then $\psi$ is a~\emph{primitive} morphism. 
By definition, we have for each $u \in {\mathcal A}^*$ the following relation for the Parikh vectors $\vec{\psi(u)}=M_\psi\vec{u}$.

Let ${\bf u}$ be an infinite word over an alphabet $\mathcal A$. Then the \emph{uniform frequency} $f_{\tt i}$ of the letter ${\tt i}\in \mathcal A$ is equal to $\alpha$ if for any sequence $(w_{n})$ of factors of ${\bf u}$ with increasing lengths 
$$\alpha=\lim_{n\to \infty}\frac{|w_{n}|_{\tt i}}{|w_{n}|}\,.$$
It is known that fixed points of primitive morphisms have uniform letter frequencies~\cite{Quef87}.

Let ${\bf u}$ be an infinite word over an alphabet $\mathcal A$ and let $\psi:{\mathcal A}^* \to {\mathcal B}^*$ be a~morphism. 
Consider a factor $w$ of $\psi({\bf u})$. We say that $(w_1, w_2)$ is a \emph{synchronization point} of $w$ if $w=w_1w_2$ and for all $p,s \in {\mathcal L}(\psi({\bf u}))$ and $v \in {\mathcal L}({\bf u})$ such that $\psi(v)=pws$ there exists a factorization $v=v_1v_2$ of $v$ with $\psi(v_1)=pw_1$ and $\psi(v_2)=w_2s$. We denote the synchronization point by $w_1\bullet w_2$.


Given a factorial language $L$ and an integer $\ell$, let $L^\ell$ denote the words of length $\ell$ in $L$.
The \emph{Rauzy graph} of $L$ of order $\ell$ is the directed graph
whose vertices are the words of $L^{\ell-1}$, the arcs are the words of $L^{\ell}$, and the arc
corresponding to the word $w$ goes from the vertex corresponding to the prefix of $w$ of length $\ell -1$
to the vertex corresponding to the suffix of $w$ of length $\ell -1$.

Finally, this paper mainly studies properties of the words $\mu(\bf p)$ and $\nu(\bf p)$ that are morphic images
of the word ${\bf p}=\varphi^\omega(\texttt{0})$ studied in~\cite{CORS2022}, where

\noindent
\begin{minipage}[b]{0.3\linewidth}
\centering
\begin{align*}
    \varphi(\texttt{0}) &= \texttt{01}\\
    \varphi(\texttt{1}) &= \texttt{21}\\
    \varphi(\texttt{2}) &= \texttt{0}
\end{align*}
\end{minipage}
\begin{minipage}[b]{0.3\linewidth}
\centering
\begin{align*}
    \mu(\texttt{0}) &= \texttt{011001}\\
    \mu(\texttt{1}) &= \texttt{1001}\\
    \mu(\texttt{2}) &= \texttt{0}
\end{align*}
\end{minipage}
\begin{minipage}[b]{0.3\linewidth}
\centering
\begin{align*}
    \nu(\texttt{0}) &= \texttt{011}\\
    \nu(\texttt{1}) &= \texttt{0}\\
    \nu(\texttt{2}) &= \texttt{01}
\end{align*}
\end{minipage}

\section{Fewest palindromes, least critical exponent, and factor complexity}\label{sec:pal}
\subsection{General result}
\begin{theorem}\label{thm:pairs}
There exists an infinite binary $\beta^+$-free word containing only $p$ palindromes
for the following pairs $(p, \beta)$. Moreover, this list of pairs is optimal.
\begin{enumerate}[label=(\alph*)]
\item $(11,\tfrac{10}3)$\label{r11}
\item $(12,\tfrac{23}7)$\label{r12}
\item $(13,3)$\label{r13}
\item $(15,\tfrac83)$\label{r15}
\item $(18,\tfrac{28}{11})$\label{r18}
\item $(20,\tfrac52)$\label{r20}
\item $(25,\tfrac73)$\label{r25}
\end{enumerate}
\end{theorem}

\begin{proof}
The optimality is obtained by backtracking. For example, the step between \cref{r13,r15} is obtained by showing that
there exists no infinite cubefree word containing at most $14$ palindromes.
The proof of the positive results is split in two cases, depending on the factor complexity of the considered words,
see~\Cref{thm:exp_pairs,thm:fc}. Let us already remark that, in any case, it is easy to
check that the proposed word does not contain more than the claimed number of palindromes
since it only requires to check the factors up to some finite length.
\end{proof}

\subsection{Exponential cases}
We need some terminology and a lemma from~\cite{Mol&Rampersad&Shallit:2020}.
A morphism $f:\Sigma^*\rightarrow\Delta^*$ is \emph{$q$-uniform\/} if $|f(a)|=q$ for every $a\in\Sigma$, and is called \emph{synchronizing} if for all $a,b,c\in\Sigma$ and $u,v\in \Delta^*$, if $f(ab)=uf(c)v$, then either $u=\varepsilon$ and $a=c$, or $v=\varepsilon$ and $b=c$.
\begin{lemma}~\cite[Lemma 23]{Mol&Rampersad&Shallit:2020}\label{lem:MRS}
Let $a,b\in\mathbb{R}$ satisfy $1<a<b$. Let $\alpha\in\{a,a^+\}$ and $\beta\in\{b,b^+\}$.
Let $h\colon \Sigma^*\rightarrow \Delta^*$ be a synchronizing $q$-uniform morphism. Set
$$t = \max\paren{\frac{2b}{b-a},\frac{2(q-1)(2b-1)}{q(b-1)}}.$$
If $h(w)$ is $\beta$-free for every $\alpha$-free word $w$ with
$|w|\leq t$, then $h(z)$ is $\beta$-free for every $\alpha$-free word $z \in \Sigma^*$.
\label{mrs}
\end{lemma}
The results in this subsection use the following steps.
We find an appropriate uniform synchronizing morphism $h$ by exhaustive search.
We use \Cref{lem:MRS} to show that $h$ maps every binary $\tfrac73^+$-free word
(resp. ternary squarefree word) to a suitable binary $\beta^+$-free word.
Since there are exponentially many binary $\tfrac73^+$-free words~\cite{KS03}
(resp. ternary squarefree words~\cite{Shur:2012}), there are also exponentially many binary $\beta^+$-free words.

\begin{theorem}\label{thm:exp_pairs}
There exist exponentially many infinite binary $\beta^+$-free words containing at most $p$ palindromes
for the following pairs $(p, \beta)$.
\begin{enumerate}[label=(\alph*)]
\item $(11,\tfrac{10}3)$\label{e11}
\item $(12,\tfrac{23}7)$\label{e12}
\item $(13,3)$\label{e13}
\item $(15,\tfrac83)$\label{e15}
\item $(18,\tfrac{13}5)$\label{e18}
\item $(19,\tfrac{28}{11})$\label{e19}
\item $(21,\tfrac52)$\label{e21}
\item $(25,\tfrac73)$\label{e25}
\end{enumerate}
\end{theorem}

\begin{proof}{\ }
\begin{enumerate}[label=(\alph*)]
\item $(11,\tfrac{10}3)$: Applying the 39-uniform morphism
\begin{align*}
\texttt{0}&\rightarrow\texttt{001011001011100101110010110010111001011} \\
\texttt{1}&\rightarrow\texttt{100101100101100101110010110010111001011}
\end{align*}
to any binary $\tfrac73^+$-free word gives a $\tfrac{10}3^+$-free binary word containing at most $11$ palindromes.
\item $(12,\tfrac{23}7)$: Applying the 45-uniform morphism
\begin{align*}
\texttt{0}&\rightarrow\texttt{000101100010111000101110001011000101110001011} \\
\texttt{1}&\rightarrow\texttt{100010110001011000101110001011000101110001011}
\end{align*}
to any binary $\tfrac73^+$-free word gives a $\tfrac{23}7^+$-free binary word containing at most $12$ palindromes.
\item $(13,3)$: Applying the 7-uniform morphism
\begin{align*}
\texttt{0}&\rightarrow\texttt{0001011} \\
\texttt{1}&\rightarrow\texttt{1001011}
\end{align*}
to any binary $\tfrac73^+$-free word gives a cubefree binary word containing at most $13$ palindromes.
\item $(15,\tfrac83)$: Applying the 3-uniform morphism
\begin{align*}
\texttt{0}&\rightarrow\texttt{001} \\
\texttt{1}&\rightarrow\texttt{101}
\end{align*}
to any binary $\tfrac73^+$-free word gives a $\tfrac83^+$-free binary word containing at most $15$ palindromes.
\item $(18,\tfrac{13}5)$: Applying the 72-uniform morphism
\begin{align*}
\texttt{0}&\rightarrow\texttt{001011001100101100101001011001100101100110010100101100101001011001100101} \\
\texttt{1}&\rightarrow\texttt{100110010100101100110010110010100101100110010100101100101001011001100101}
\end{align*}
to any binary $\tfrac73^+$-free word gives a $\tfrac{13}5^+$-free binary word containing at most $18$ palindromes.
\item $(19,\tfrac{28}{11})$: Applying the 49-uniform morphism
\begin{align*}
\texttt{0}&\rightarrow\texttt{0010110010110101100101001011001010010110101100101} \\
\texttt{1}&\rightarrow\texttt{1010110010100101100101101011001010010110101100101}
\end{align*}
to any binary $\tfrac73^+$-free word gives a $\tfrac{28}{11}^+$-free binary word containing at most $19$ palindromes.
\item $(21,\tfrac52)$: Applying the 10-uniform morphism
\begin{align*}
\texttt{0}&\rightarrow\texttt{0011001101} \\
\texttt{1}&\rightarrow\texttt{1001011001}
\end{align*}
to any binary $\tfrac73^+$-free word gives a $\tfrac52^+$-free binary word containing at most $21$ palindromes.
\item $(25,\tfrac73)$: Applying the 36-uniform morphism
\begin{align*}
\texttt{0}&\rightarrow\texttt{001101100101100110110010011001011001} \\
\texttt{1}&\rightarrow\texttt{101100100110100110110010011001011001} \\
\texttt{2}&\rightarrow\texttt{001101100110100110110010011001011001}
\end{align*}

to any ternary squarefree word gives a $\tfrac73^+$-free binary word containing at most $25$ palindromes.
\end{enumerate}

\end{proof}

\subsection{Polynomial cases}
\begin{theorem}~\cite{CORS2022}\label{lemmap}
Every bi-infinite ternary cubefree word avoiding
 $$F=\textnormal{\{\texttt{00},\texttt{11},\texttt{22},\texttt{20},\texttt{212},\texttt{0101},\texttt{02102},\texttt{121012},\texttt{01021010},\texttt{21021012102}}\}$$
has the same set of factors as $\bf p$.
\end{theorem}

\begin{lemma}\label{lemmamu}
Every bi-infinite cubefree binary word avoiding
$$F_{18}=\textnormal{\{\texttt{1101},\texttt{00100},\texttt{10101},\texttt{010011},\texttt{1011001011},\texttt{110010110011},\texttt{1011001010010110010}\}}$$
has the same set of factors as $\mu(\bf p)$.
\end{lemma}
\begin{proof}
Consider a bi-infinite binary cubefree word $\bf w$ avoiding $F_{18}$.\\
The factors of $\bf w$ of length at least 5 that contain \texttt{0101} only as a prefix and a suffix
are \texttt{01011001100101}, \texttt{010100101}, and \texttt{0101100101}.
Thus, $\bf w$ is in $\acc{\texttt{0110011001},\texttt{01001},\texttt{011001}}^\omega$.
So, $\bf w$ is in $\acc{\texttt{011001},\texttt{1001},\texttt{0}}^\omega$.
That is, $\bf w=\mu(\bf v)$ for some bi-infinite ternary word~$\bf v$.
Since $\bf w$ is cubefree, its pre-image $\bf v$ is also cubefree.

To show that $\bf v$ avoids $F$, we consider every $f\in F$
and we show by contradiction that $f$ is not a factor of $\bf v$.
\begin{enumerate}[label=(\alph*)]
 \item If $\bf v$ contains $\texttt{22}$, then $\mu(\texttt{220})=\texttt{00011001}$ and $\mu(\texttt{222})=\texttt{000}$ contain $\texttt{0}^3$ and $\mu(\texttt{221})=\texttt{001001}$ contains $\texttt{00100}\in F_{18}$. \label{22}
 \item If $\bf v$ contains $\texttt{20}$, then $\bf v$ contains $x\texttt{20}$ for $x\in\acc{\texttt{0},\texttt{1}}$ by \ref{22}.\\
  $\mu(x\texttt{20})$ contains \texttt{10010011001} as a suffix, which contains $\texttt{00100}\in F_{18}$. \label{20}
 \item If $\bf v$ contains $\texttt{00}$, then $\bf v$ contains $\texttt{100}$ to avoid the cube $\texttt{000}$ and by \ref{20}.\\
  $\mu(\texttt{100})=\texttt{1001011001011001}$ contains $\texttt{1011001011}\in F_{18}$. \label{00}
 \item If $\bf v$ contains $\texttt{11}$, then $\mu(\texttt{011})$ and $\mu(\texttt{111})$ contain $(\texttt{1001})^3$
 and $\mu(\texttt{211})=\texttt{010011001}$ contains $\texttt{010011}\in F_{18}$. \label{11}
 \item If $\bf v$ contains $\texttt{212}$, then $\bf v$ contains $\texttt{2121}$ by \ref{22} and \ref{20}.\\
  $\bf v$ contains $x\texttt{2121}y$ with $x\in\acc{\texttt{0},\texttt{1}}$ by \ref{22} and $y\in\acc{\texttt{0},\texttt{2}}$ by \ref{11}.\\
  Since $\mu(\texttt{1})$ is a suffix of $\mu(\texttt{0})$ and $\mu(\texttt{2})$ is a prefix of $\mu(\texttt{0})$, 
  then $\mu(x\texttt{2121}y)$ contains the factor $\mu(\texttt{121212})=\mu\paren{(\texttt{12})^3}$.\label{212}
 \item If $\bf v$ contains $\texttt{0101}$, then $\mu(\texttt{0101})=\texttt{01100110010110011001}$ contains $\texttt{110010110011}\in F_{18}$. \label{0101}
 \item If $\bf v$ contains $\texttt{02102}$, then $\bf v$ contains $\texttt{102102}$ by \ref{20} and \ref{00}.\\
  $\mu(\texttt{102102})=\texttt{1001011001010010110010}$ contains $\texttt{1011001010010110010}\in F_{18}$.\label{02102}
 \item If $\bf v$ contains $\texttt{121012}$, then $\bf v$ contains $\texttt{0121012}$ by \ref{11} and \ref{212}.\\
  $\bf v$ contains $\texttt{10121012}$ by \ref{20} and \ref{00}.\\
  $\bf v$ contains $\texttt{210121012}$ by \ref{11} and \ref{0101}.\\
  $\bf v$ contains $\texttt{2101210121}$ by \ref{22} and \ref{20}.\\
  $\bf v$ contains $\texttt{21012101210}$ by \ref{11} and \ref{212}.\\
  $\bf v$ contains $x\texttt{21012101210}$ with $x\in\acc{\texttt{0},\texttt{1}}$ by \ref{22}.\\
  Since $\mu(\texttt{1})$ is a suffix of $\mu(\texttt{0})$, then
  $\mu(x\texttt{21012101210})$ contains $\mu(\texttt{121012101210})=\mu\paren{(\texttt{1210})^3}$.\label{121012}
 \item If $\bf v$ contains \texttt{01021010}, then $\bf v$ contains \texttt{010210102} by \ref{00} and \ref{0101}.\\
 $\bf v$ contains \texttt{0102101021} by \ref{22} and \ref{20}.\\
 $\bf v$ contains \texttt{01021010210} by \ref{11} and \ref{212}.\\
 $\bf v$ contains \texttt{010210102101} by \ref{00} and \ref{02102}.\\
 $\bf v$ contains \texttt{1010210102101} by \ref{20} and \ref{00}.\\
 $\bf v$ contains \texttt{21010210102101} by \ref{11} and \ref{0101}.\\
 $\bf v$ contains \texttt{210102101021012} by \ref{11} and to avoid $(\tt 21010)^3$.\\
 $\bf v$ contains \texttt{1210102101021012} by \ref{22} and to avoid $(\tt 02101)^3$.\\
 $\bf w$ contains $\mu(\texttt{1210102101021012})=\texttt{1001}\mu(\texttt{210})\texttt{1001011001}\mu(\texttt{210})\texttt{1001011001}\mu(\texttt{210})\texttt{10010}$.\\
 To avoid $\texttt{00100}\in F_{18}$, $\bf w$ contains\\
 $\texttt{1001}\mu(\texttt{210})\texttt{1001011001}\mu(\texttt{210})\texttt{1001011001}\mu(\texttt{210})\texttt{100101}=(\texttt{1001}\mu(\texttt{210})\texttt{100101})^3$.
 \item If $\bf v$ contains $\texttt{21021012102}$, then  $\bf v$ contains $\texttt{121021012102}$ by \ref{22} and \ref{02102}.\\
  $\bf v$ contains $\texttt{0121021012102}$ by \ref{11} and \ref{212}.\\ 
  $\bf v$ contains $\texttt{10121021012102}$ by \ref{20} and \ref{00}.\\
  $\bf v$ contains $\texttt{210121021012102}$ by \ref{11} and \ref{0101}.\\
  $\bf v$ contains $\texttt{0210121021012102}$ by \ref{22} and \ref{121012}.\\
  $\bf v$ contains $\texttt{10210121021012102}$ by \ref{20} and \ref{00}.\\
  $\bf v$ contains $\texttt{102101210210121021}$ by \ref{22} and \ref{20}.\\
  $\bf v$ contains $\texttt{1021012102101210210}$ by \ref{11} and \ref{212}.\\
  $\bf v$ contains $\texttt{10210121021012102101}$ by \ref{00} and \ref{02102}.\\
  $\bf v$ contains $\texttt{102101210210121021010}$ by \ref{11} and to avoid $(\tt 1021012)^3$.\\
  $\mu(\texttt{102101210210121021010})=(\mu(\texttt{102101})\texttt{0})^3\texttt{11001}$.
\end{enumerate}
\end{proof}

\begin{lemma}\label{lemmanu}
Every bi-infinite cubefree binary word avoiding
$$F_{20}=\textnormal{\{\texttt{0101},\texttt{1011},\texttt{010010},\texttt{1100110100110011}\}}$$
has the same set of factors as $\nu(\bf p)$.
\end{lemma}
\begin{proof}
Consider a bi-infinite binary cubefree word $\bf w$ avoiding $F_{20}$.
Since $\bf w$ is cubefree, $\bf w$ is in $\acc{\texttt{011},\texttt{0},\texttt{01}}^\omega$.
So $\bf w=\nu(\bf v)$ for some bi-infinite ternary word~$\bf v$.
Since $\bf w$ is cubefree, its pre-image $\bf v$ is also cubefree.

To show that $\bf v$ avoids $F$, we consider every $f\in F$
and we show by contradiction that $f$ is not a factor of $\bf v$.
\begin{enumerate}[label=(\alph*)]
 \item If $\bf v$ contains $\texttt{00}$, then $\nu(\texttt{00})=\texttt{011011}$ contains $\texttt{1011}\in F_{20}$.\label{00b}
 \item If $\bf v$ contains $\texttt{11}$, then $\bf v$ contains $\texttt{11}y$ for some letter $y$.\\
 $\nu(\texttt{11}y)$ contains the cube \texttt{000} as a prefix.\label{11b}
 \item If $\bf v$ contains $\texttt{22}$, then $\nu(\texttt{22})=\texttt{0101}\in F_{20}$.\label{22b}
 \item If $\bf v$ contains $\texttt{20}$, then $\nu(\texttt{20})=\texttt{01011}$ contains $\texttt{0101}\in F_{20}$.\label{20b}
 \item If $\bf v$ contains $\texttt{212}$, then $\bf v$ contains $\texttt{2121}$ by \ref{22b} and \ref{20b}.\\ \label{212b}
 $\nu(\texttt{2121})=\texttt{010010}\in F_{20}$.
 \item If $\bf v$ contains $\texttt{0101}$, then  $\bf v$ contains $\texttt{10101}$ by \ref{00b} and \ref{20b}.\\
  $\bf v$ contains $\texttt{210101}$ by \ref{11b} and to avoid $(\texttt{01})^3$.\\
  $\bf v$ contains $\texttt{2101012}$ by \ref{11b} and to avoid $(\texttt{10})^3$.\\
 $\nu(\texttt{2101012})=\texttt{0100110011001}=\texttt{0}(\texttt{1001})^3$. \label{0101b}
 \item If $\bf v$ contains $\texttt{02102}$, then $\bf v$ contains $\texttt{102102}$ by \ref{00b} and \ref{20b}.\\
  $\bf v$ contains $\texttt{1021021}$ by \ref{22b} and \ref{20b}.\\
  $\bf v$ contains $\texttt{10210210}$ by \ref{11b} and \ref{212b}.\\
  $\bf w$ contains $\nu(\texttt{10210210})=\texttt{0011010011010011}$.\\
  To avoid $\texttt{0}^3$ and $\texttt{1}^3$, $\bf w$ contains $\texttt{100110100110100110}=(\texttt{100110})^3$.\label{02102b}
 \item If $\bf v$ contains $\texttt{121012}$, then $\bf v$ contains $\texttt{0121012}$ by \ref{11b} and \ref{212b}.\\
  $\bf v$ contains $\texttt{10121012}$ by \ref{00b} and \ref{20b}.\\
  $\bf v$ contains $\texttt{210121012}$ by \ref{11b} and \ref{0101b}.\\
  $\bf v$ contains $\texttt{2101210121}$ by \ref{22b} and \ref{20b}.\\
  $\bf v$ contains $\texttt{21012101210}$ by \ref{11b} and \ref{212b}.\\
  $\bf w$ contains $\nu(\texttt{21012101210})=\texttt{01001100100110010011}$.\\
  To avoid $\texttt{1}^3$, $\bf w$ contains $\texttt{010011001001100100110}=(\texttt{0100110})^3$.\label{121012b}
 \item If $\bf v$ contains $\texttt{01021010}$, then\\
  $\nu(\texttt{01021010})=\texttt{01100110100110011}$ contains $\texttt{1100110100110011}\in F_{20}$.
 \item If $\bf v$ contains $\texttt{21021012102}$, then  $\bf v$ contains $\texttt{121021012102}$ by \ref{22b} and \ref{02102b}.\\
  $\bf v$ contains $\texttt{0121021012102}$ by \ref{11b} and \ref{212b}.\\
  $\bf v$ contains $\texttt{10121021012102}$ by \ref{00b} and \ref{20b}.\\
  $\bf v$ contains $\texttt{210121021012102}$ by \ref{11b} and \ref{0101b}.\\
  $\bf v$ contains $\texttt{0210121021012102}$ by \ref{22b} and \ref{121012b}.\\
  $\bf v$ contains $\texttt{10210121021012102}$ by \ref{00b} and \ref{20b}.\\
  $\bf v$ contains $\texttt{102101210210121021}$ by \ref{22b} and \ref{20b}.\\
  $\bf v$ contains $\texttt{1021012102101210210}$ by \ref{11b} and \ref{212b}.\\
  $\bf v$ contains $\texttt{10210121021012102101}$ by \ref{00b} and \ref{02102b}.\\
  $\bf v$ contains $\texttt{102101210210121021010}$ by \ref{11b} and to avoid $(\tt 1021012)^3$.\\
 $\nu(\texttt{102101210210121021010})=(\texttt{0011010011001})^3\texttt{1}$.
\end{enumerate}
\end{proof}

\begin{theorem}\label{thm:fc}{\ }
\begin{enumerate}[label=(\alph*)]
\item The word $\mu({\bf p})$ is $\tfrac{28}{11}^+$-free and contains $18$ palindromes.
Every bi-infinite $\tfrac{13}5$-free binary word containing at most $18$ palindromes
has the same set of factors as either
$\mu({\bf p})$, $\overline{\mu({\bf p})}$, $\mu({\bf p})^R$, or $\overline{\mu({\bf p})^R}$.
\item The word $\nu({\bf p})$ is $\tfrac52^+$-free and contains $20$ palindromes.
Every recurrent $\tfrac{28}{11}$-free binary word containing at most $20$ palindromes
has the same set of factors as either
$\nu({\bf p})$, $\overline{\nu({\bf p})}$, $\nu({\bf p})^R$, or $\overline{\nu({\bf p})^R}$. 
\end{enumerate}
\end{theorem}

\begin{proof}{\ }
\begin{enumerate}[label=(\alph*)]
\item We prove in Section~\ref{sec:mu} that $\mu({\bf p})$ is $\tfrac{28}{11}^+$-free.

We construct the set $S_{18}^{20}$ defined as follows:
a word $v$ is in $S_{18}^{20}$ if and only if there exists a $\tfrac{13}5$-free binary word $pvs$ containing
at most $18$ palindromes and such that $|p|=|v|=|s|=20$.
From $S_{18}^{20}$, we construct the Rauzy graph $R_{18}^{20}$ such that the vertices are the factors
of length 19 and the arcs are the factors of length 20.
We notice that $R_{18}^{20}$ is disconnected. It contains four connected components that are symmetric
with respect to reversal and bit complement. Let $C_{18}^{20}$ be the connected component which avoids the factor \texttt{1101}.
We check that $C_{18}^{20}$ is identical to the Rauzy graph of the factors of length 19 and 20 of $\mu(\bf p)$.

Now we consider a bi-infinite $\tfrac{13}5$-free binary word $\bf w$ with $18$ palindromes.
So $\bf w$ corresponds to a walk in one of the connected components of $R_{18}^{20}$,
say $C_{18}^{20}$ without loss of generality.
By the previous remark, $\bf w$ has the same set of factors of length 20 as $\mu(\bf p)$.
Since $\max\acc{|f|, f\in F_{18}}=19\le20$, $\bf w$ avoids every factor in $F_{18}$.
Moreover, $\bf w$ is cubefree since it is $\tfrac{13}5$-free.
By \Cref{lemmamu}, $\bf w$ has the same factor set as $\mu(\bf p)$.

Then the proof is complete by symmetry by reversal and bit complement.

\item We prove in Section~\ref{sec:nu} that $\nu({\bf p})$ is $\tfrac52^+$-free.

We construct the set $S_{20}^{78}$ defined as follows:
a word $v$ is in $S_{20}^{78}$ if and only if there exists a $\tfrac{28}{11}$-free binary word $pvs$ containing
at most $20$ palindromes and such that $|p|=|v|=|s|=78$.
From $S_{20}^{78}$, we construct the Rauzy graph $R_{20}^{78}$ such that the vertices are the factors
of length 77 and the arcs are the factors of length 78.
We notice that $R_{20}^{78}$ is not strongly connected.
It contains four strongly connected components that are symmetric with respect to reversal and bit complement.
Let $C_{20}^{78}$ be the strongly connected component which avoids the factor \texttt{1011}.
We check that $C_{20}^{78}$ is identical to the Rauzy graph of the factors of length 77 and 78 of $\nu({\bf p})$.

Now we consider a recurrent $\tfrac{28}{11}$-free binary word $\bf w$ with $20$ palindromes.
Since $\bf w$ is recurrent, $\bf w$ corresponds to a walk in one of the strongly connected
components of $R_{20}^{78}$, say $C_{20}^{78}$ without loss of generality.
By the previous remark, $\bf w$ has the same set of factors of length 78 as $\nu(\bf p)$.
Since $\max\acc{|f|, f\in F_{20}}=16\le78$, $\bf w$ avoids every factor in $F_{20}$.
Moreover, $\bf w$ is cubefree since it is $\tfrac{28}{11}$-free.
By \Cref{lemmanu}, $\bf w$ has the same factor set as $\nu({\bf p})$.

Then the proof is complete by symmetry by reversal and bit complement.

\end{enumerate}
\end{proof}







Notice that item (b) requires recurrent words rather than bi-infinite words.
That is because of, e.g., the bi-infinite word ${\bf x}=\nu({\bf p})^R\texttt{010110}\nu({\bf p})$.
Obviously $\nu({\bf p})$ and $\nu({\bf p})^R$ have the same set of 20 palindromes and it is easy
to check that ${\bf x}$ contains no additional palindrome.
We show that ${\bf x}$ is $\tfrac52^+$-free by checking the central factor
of ${\bf x}$ of length 200. Then larger repetitions of exponent $>\tfrac52$ are ruled out since
the word \texttt{110011001001101} is a prefix of $\texttt{110}\nu(\bf p)$ but is neither a factor of $\nu(\bf p)$ nor $\nu(\bf p)^R$.
By symmetry, this also holds for ${\bf x}^R=\nu({\bf \bf p})^R\texttt{011010}\nu({\bf p})$,
$\overline{\bf x}$, and $\overline{{\bf x}^R}$.
\section{The critical exponent of $\nu(\bf p)$ and $\mu(\bf p)$}\label{sec:mu_nu}
\bigskip

Before recalling the definition of the infinite words $\bf p$, $\nu(\bf p)$ and $\mu(\bf p)$, let us underline that all of them are uniformly recurrent and $\nu(\bf p)$ and $\mu(\bf p)$ are morphic images of $\bf p$. Hence in order to compute their (asymptotic) critical exponents, we will exploit the following two useful statements. See also~\cite{diploma2024}.

\begin{theorem}[\cite{DDP21}] \label{thm:formulaE}
Let $\bf{u}$ be a uniformly recurrent aperiodic infinite word. Let $(w_n)$ be a sequence of all bispecial factors ordered by their length.
For every $n\in\mathbb{N}$, let $r_n$ be a shortest return word to $w_n$ in $\bf{u}$. Then
\begin{equation}    
{E}(\mathbf{u}) = 1 + \sup\limits_{n \in \mathbb{N}} \left\{ \frac{|w_n|}{|r_n|} \right\}
\qquad \text{and} \qquad
{E}^*(\mathbf{u}) = 1 + \limsup\limits_{n \to +\infty} \frac{|w_n|}{|r_n|}\,.
\end{equation}
\end{theorem}


\begin{theorem}\label{thm:E*_morphic_image} Let ${\bf u}$ be an infinite word over an alphabet $\mathcal A$ such that the uniform letter frequencies in ${\bf u}$ exist. Let $\psi:{\mathcal A}^* \to {\mathcal B}^*$ be an injective morphism and let $L \in \mathbb N$ be such that every factor $v$ of $\psi({\bf u})$, $|v|\geq L$, has a~synchronization point. 
Then $E^*({{\bf u}})=E^*(\psi({{\bf u}}))$.
\end{theorem}
\begin{proof}
The inequality $E^*(\psi({{\bf u}}))\geq E^*({{\bf u}})$ is proven in~\cite{DvPe2023} for any non-erasing morphism under the assumption of existence of uniform letter frequencies in ${\bf u}$.
Let us prove the opposite inequality. According to the definition of $E^*(\psi({\bf u}))$, there exist sequences $\bigl(w_{n}\bigr)$ and $\bigl(v_{n}\bigr)$ such that 
\begin{enumerate}
\item $\lim\limits_{n\to \infty} |v_{n}| = \infty$; 
\item $w_{n}$ is a factor of $\psi({\bf u})$ for each $n \in \mathbb N$; 
\item $w_{n}$ is a prefix of the periodic word $\bigl(v_{n}\bigr)^\omega$ for each $n \in \mathbb N$; 
\item $E^*(\psi({{\bf u}}))=\lim\limits_{n\to \infty}\frac{|w_{n}|}{|v_{n}|}$.
\end{enumerate}
If $E^*(\psi({{\bf u}}))=1$, then, clearly, $E^*(\psi({{\bf u}}))\leq E^*({{\bf u}})$.
Assume in the sequel that $E^*(\psi({{\bf u}})) > 1$, then we have for large enough $n$ that $|w_{n}|>|v_{n}|$ and moreover, by the first item, $|v_n|\geq L$. By assumption, both $v_n$ and $w_n$ have synchronization points and since $v_n$ is a prefix of $w_n$ for large enough $n$, we may write $$w_{n}=x_{n}\bullet\psi(w'_{n})\bullet y_{n}\quad  \text{and} \quad v_{n}=x_{n}\bullet \psi(v'_{n})\bullet z_{n}\,,$$ where we highlighted the first and the last synchronization point (not necessarily distinct) in $w_n$ and $v_n$ and where $w'_{n}$ and $v'_{n}$ are uniquely given factors of ${\bf u}$ and the lengths of $x_{n}$, $y_{n}$, $z_{n}$ are smaller than $L$.  

\medskip
By the third item, we have $$w_{n}=v_{n}^ku_n=(x_{n}\psi(v'_{n})z_{n})^k u_{n}\,,$$ 
where $u_{n}$ is a proper prefix of $v_{n}$ and $k\in \mathbb N, k\geq 1$.

There are two possible cases for $(u_{n})$.
\begin{enumerate}
\item[(a)] Either $(|u_{n}|)$ is bounded, but as $E^*(\psi({{\bf u}})) > 1$, it follows that $k\geq 2$ for large enough $n$. 
\item[(b)] Or there is a subsequence $(u_{j_n})$ of $(u_{n})$ such that for all $n \in \mathbb N$ we have $|u_{j_n}|\geq L$. Then by assumption, $u_{j_n}$ has a~synchronization point and we may write $u_{j_n}=x_{j_n}\bullet\psi(u'_{j_n})\bullet y_{j_n}$, where we highlighted the first and the last synchronization point in $u_{j_n}$ and $u'_{j_n}$ is a prefix of $v'_{j_n}$ by injectivity of $\psi$. 
\end{enumerate}
\begin{enumerate}
\item[(a)]  In the first case, since $k \geq 2$ for large enough $n$, the factor $w_{n}$ starts with $(x_{n}\psi(v'_{n})z_{n})^2$. By definition of synchronization points and injectivity of $\psi$, there exists a~unique factor $t_{n}$ of ${\bf u}$ such that $\psi(v'_{n})z_{n}x_{n}=\psi(t_{n})$. Consequently, $w_{n}=(x_{n}\psi(v'_{n})z_{n})^k u_{n}=x_n\psi(t_n^{k-1}v'_n)z_n u_n$. Therefore, $t_n^{k-1}v'_n$ is a factor of ${\bf u}$ and it is a prefix of $(t_n)^{\omega}$ and $$E^*(\psi({{\bf u}}))=\lim\limits_{n\to \infty}\frac{|w_{n}|}{|v_{n}|}=\lim\limits_{n\to \infty}\frac{|x_{n}\psi({t_{n}}^{k-1}v'_n) z_{n}u_{n}|}{|\psi(t_{n})|}=\lim\limits_{n\to \infty}\frac{|\psi({t_{n}}^{k-1}v'_n)|}{|\psi(t_{n})|}\,,$$
where the last equality holds thanks to boundedness of $(|x_n|), (|z_n|)$ and $(|u_n|)$. 
\item[(b)] In the second case, $w_{j_n}=(x_{j_n}\psi(v'_{j_n})z_{j_n})^k x_{j_n}\psi(u'_{j_n})y_{j_n}$, where $k\geq 1$. By definition of synchronization points and injectivity of $\psi$, there exists a~unique factor $t_{j_n}$ of ${\bf u}$ such that $\psi(v'_{j_n})z_{j_n}x_{j_n}=\psi(t_{j_n})$. Consequently, $(t_{j_n})^{k}u'_{j_n}$ is a factor of ${\bf u}$ and it is a prefix of $(t_{j_n})^{\omega}$ and $$E^*(\psi({{\bf u}}))=\lim\limits_{n\to \infty}\frac{|w_{n}|}{|v_{n}|}=\lim\limits_{n\to \infty}\frac{|x_{j_n}\psi\bigl(({t_{j_n}})^k u'_{j_n}\bigr)y_{j_n}|}{|\psi(t_{j_n})|}=\lim\limits_{n\to \infty}\frac{|\psi\bigl(({t_{j_n}})^k u'_{j_n}\bigr)|}{|\psi(t_{j_n})|}\,,$$ 
where the last equality holds thanks to boundedness of $(x_n)$ and $(y_n)$.
\end{enumerate}

\medskip
Combining two simple facts: 
\begin{itemize}
    \item $\frac{|\psi(u)|}{|u|}={\vec 1 \hspace{0.01cm}}^{T} M_\psi \frac{\vec u}{|u|}$ for each word $u$ over $\mathcal A$, where $\vec 1$ is a vector with all coordinates equal to one;
    \item for each sequence $(s_n)$ of factors of ${\bf u}$ with $\lim_{n \to \infty} |s_n|=\infty$ we have, by uniform letter frequencies in ${\bf u}$, $\lim\limits_{n\to \infty}\frac{\vec{s_{n}}}{|s_{n}|}= \vec f\,$, where $\vec f$ is the vector of letter frequencies in ${\bf u}$,
    \end{itemize}
    we obtain
\begin{equation}\label{eq:psi}
\lim_{n \to \infty}\frac{|\psi(s_{n})|}{|s_{n}|} ={\vec 1 \hspace{0.01cm}}^{T} M_\psi \vec f \,.
\end{equation}
Consequently, 
\begin{enumerate}
\item[(a)]
in the first case, since $\lim_{n \to \infty} |t_n|=\infty$, we obtain using~\eqref{eq:psi}

$$\begin{array}{rcl}
E^*(\psi({{\bf u}}))&=&\lim\limits_{n\to \infty}\frac{|\psi({t_{n}}^{k-1}v'_n)|}{|\psi(t_{n})|}\\
&=&\lim\limits_{n\to \infty}\frac{|\psi({t_{n}}^{k-1}v'_n)|}{|{t_{n}}^{k-1}v'_n|}\frac{|t_n|}{|\psi(t_{n})|}\frac{|{t_{n}}^{k-1}v'_n|}{|t_n|}\\
&=&\lim\limits_{n\to \infty}\frac{|{t_{n}}^{k-1}v'_n|}{|t_n|}\leq E^*({\bf u})\,,
\end{array}$$
where the last inequality follows from the fact that 
${(t_{n})}^{k-1} v'_{n}\in {\mathcal L}({\bf u})$ and ${(t_{n})}^{k-1} v'_{n}$ is a power of $t_{n}$;

\item[(b)] in the second case, since $\lim |t_{j_n}|=\infty$, we obtain using~\eqref{eq:psi}
$$\begin{array}{rcl}
E^*(\psi({{\bf u}}))&=&\lim\limits_{n\to \infty}\frac{|\psi\bigr({(t_{j_n})}^k u'_{j_n}\bigl)|}{|\psi(t_{j_n})|}\\
&=&\lim\limits_{n\to \infty}\frac{|\psi\bigl({(t_{j_n}})^k u'_{j_n}\bigr)|}{|{(t_{j_n}})^k u'_{j_n}|}\frac{|t_{j_n}|}{|\psi(t_{j_n})|}\frac{|({t_{j_n}})^k u'_{j_n}|}{|t_{j_n}|}\\
&=&\lim\limits_{n\to \infty}\frac{|({t_{j_n}})^k u'_{j_n}|}{|t_{j_n}|}\leq  E^*({\bf u})\,,

\end{array}
$$
where the last inequality follows from the fact that ${(t_{j_n})}^k u'_{j_n}\in {\mathcal L}({\bf u})$ and ${(t_{j_n})}^k u'_{j_n}$ is a power of $t_{j_n}$.
\end{enumerate}
\end{proof}

\subsection{The infinite word $\bf p$}\label{sec:p}
In order to compute the critical exponent of morphic images of $\bf p$, it is essential to describe bispecial factors and their return words in $\bf p$ and to determine the asymptotic critical exponent of $\bf p$. 

\medskip

The infinite word $\bf p$ is the fixed point of the injective morphism $\varphi$, where
\begin{align*}
    \varphi(\texttt{0}) &= \texttt{01},\\
    \varphi(\texttt{1}) &= \texttt{21},\\
    \varphi(\texttt{2}) &= \texttt{0}.  
\end{align*}
Therefore, $\bf p$ has the following prefix
$${\bf p} = \texttt{01210210102101210102101210210121010} \cdots$$
\begin{remark}
It is readily seen that each non-empty factor of $\bf p$ has a~synchronization point.
\end{remark}

The following characteristics of $\bf p$ are known~\cite{CORS2022}:
\begin{itemize}
    \item The factor complexity of $\bf p$ is $C(n) = 2n + 1$.
    \item The word $\bf p$ is not closed under reversal: $\texttt{02} \in \LL(\bf p)$, but $\texttt{20} \notin \LL(\bf p)$.
    \item The word $\bf p$ is uniformly recurrent and $\bf p$ has uniform letter frequencies because $\varphi$ is primitive.
\end{itemize}

\subsubsection{Bispecial factors in $\bf p$}
First, we will examine LS factors.
    Using the form of $\varphi$, we observe
    \begin{itemize}
        \item \texttt{0} has only one left extension: \texttt{1},
        \item \texttt{1} has two left extensions: \texttt{0} and \texttt{2},
        \item \texttt{2} has two left extensions: \texttt{0} and \texttt{1}.
    \end{itemize}
    Therefore, every LS factor has left extensions either $\{\tt 0,2\}$, or $\{\tt 0 ,1\}$. 
 \begin{lemma}\label{lem:LSp}
    Let $w \neq \varepsilon$, $w \in \LL(\bf p)$. 
    \begin{itemize}
        \item If $w$ is a LS factor such that ${\tt{0}}w, {\tt{1}} w \in \LL(\bf p)$, then ${\tt 1}\varphi(w)$ is a LS factor such that ${\tt 01}\varphi(w), {\tt 21 } \varphi(w) \in \LL(\bf p)$.
        \item If $w$ is a LS factor such that ${\tt 0} w, {\tt 2}w \in \LL(\bf p)$, then $\varphi(w)$ is a LS factor such that ${\tt 0}\varphi(w), {\tt 1}\varphi(w) \in \LL(\bf p)$.
    \end{itemize}
\end{lemma}
\begin{proof}
    It follows from the form of $\varphi$ and the fact that $\bf p$ is the fixed point of the morphism $\varphi$, i.e., if $u \in \LL(\bf p)$, then $\varphi(u)\in \LL(\bf p)$.
\end{proof}


Second, we will focus on RS factors.
    We observe
    \begin{itemize}
        \item {\tt 0} has two right extensions: {\tt 1} and {\tt 2},
        \item {\tt 1} has two right extensions: {\tt 0} and {\tt 2},
        \item {\tt 2} has only one right extension: {\tt 1}.
    \end{itemize}
    Therefore, every RS factor has right extensions either $\{\tt 1,2\}$, or $\{\tt 0,2\}$. 
Using similar arguments as for LS factors, we get the following statement.
\begin{lemma}\label{lem:RSp}
    Let $w \neq \varepsilon$, $w \in \LL(\bf p)$. 
    \begin{itemize}
        \item If $w$ is a RS factor such that $w{\tt 0}, w{\tt 2} \in \LL(\bf p)$, then $\varphi(w){\tt 0}$ is a RS factor such that $\varphi(w) {\tt 01}, \varphi(w) {\tt 02} \in \LL(\bf p)$.
        \item If $w$ is a RS factor such that $w{\tt 1}, w{\tt 2} \in \LL(\bf p)$, then $\varphi(w)$ is a RS factor such that $\varphi(w){\tt 2}, \varphi(w){\tt 0} \in \LL(\bf p)$.
    \end{itemize}
\end{lemma}


It follows from the form of LS and RS factors that we have at most 4 possible kinds of non-empty BS factors in $\bf p$. 

\begin{proposition}\label{prop:BS}
    Let  $v$ be a non-empty BS factor in $\bf p$.
    \begin{enumerate}
        \item\label{prop:BSA} ${\tt 0}v,{\tt 2}v, v{\tt 0}, v{\tt 2} \in \LL(\bf p)$ if and only if there exists $w\in \LL(\bf p)$ such that $v = {\tt 1}\varphi(w)$ and ${\tt 0}w,{\tt 1}w, w{\tt 1},w{\tt 2} \in \LL(\bf p)$.
        \item\label{prop:BSB} ${\tt 0}v,{\tt 1}v, v{\tt 1}, v{\tt 2} \in \LL(\bf p)$ if and only if there exists $w\in \LL(\bf p)$ such that $v = \varphi(w) {\tt 0}$ and ${\tt 0}w,{\tt 2}w, w{\tt 0},w{\tt 2} \in \LL(\bf p)$.
        \item\label{prop:BSC} ${\tt 0}v,{\tt 2}v, v{\tt 1}, v{\tt 2} \in \LL(\bf p)$ if and only if there exists $w\in \LL(\bf p)$ such that $v = {\tt 1}\varphi(w) {\tt 0}$ and ${\tt 0}w,{\tt 1}w, w{\tt 0},w{\tt 2} \in \LL(\bf p)$.  
        \item\label{prop:BSD} ${\tt 0}v,{\tt 1}v, v{\tt 0}, v{\tt 2} \in \LL(\bf p)$ if and only if there exists $w\in \LL(\bf p)$ such that $v = \varphi(w)$ and ${\tt 0}w,{\tt 2}w, w{\tt 1},w{\tt 2} \in \LL(\bf p)$.  
    \end{enumerate}
\end{proposition}
\begin{proof}
The implication $(\Leftarrow)$ follows from Lemmata~\ref{lem:LSp} and~\ref{lem:RSp}.
We will prove the opposite implication for Item~\ref{prop:BSA}, the other cases may be proven analogously. If $v$ is a non-empty factor such that ${\tt 0}v,{\tt 2}v, v{\tt 0}, v{\tt 2} \in \LL(\bf p)$, then $v$ necessarily starts and ends with the letter ${\tt 1}$. By the form of $\varphi$, we have the following synchronization points $v={ \tt 1}\bullet \hat v\bullet$ ($\hat v$ may be empty). Hence, by injectivity of $\varphi$, there exists a unique $w$
in $\bf p$ such that $v={\tt 1}\varphi(w)$. Thus, using again the form of $\varphi$ and the knowledge of possible right extensions, the factor $w$ is BS and ${\tt 0}w,{\tt 1}w, w{\tt 1},w{\tt 2} \in \LL(\bf p)$.
\end{proof}
We can see that the only BS factor of length one is ${\tt 1}$, it has left extensions ${\tt 0,2}$ and right extensions ${\tt 0,2}$. Applying Proposition~\ref{prop:BS} Item~\ref{prop:BSB}, we obtain that $\varphi({\tt 1}){\tt 0}$ is BS with left extensions ${\tt 0,1}$ and right extensions ${\tt 1,2}$.
Proposition~\ref{prop:BS} Item~\ref{prop:BSA} gives us that ${\tt 1}\varphi^2({\tt 1})\varphi({\tt 0})$ is BS with left extensions ${\tt 0,2}$ and right extensions ${\tt 0,2}$. This process can be iterated providing us with infinitely many BS factors:
\begin{equation}\label{eq:branch1}
\begin{array}{l}
{\tt 1} \to \varphi({\tt 1}){\tt 0} \to {\tt 1}\varphi^2({\tt 1})\varphi({\tt 0}) \to \varphi({\tt 1})\varphi^3({\tt 1})\varphi^2({\tt 0}){\tt 0} \to \\ \to {\tt 1}\varphi^2({\tt 1})\varphi^4({\tt 1})\varphi^3(\texttt{{\tt 0}})\varphi({\tt 0})  \to \varphi({\tt 1})\varphi^3({\tt 1})\varphi^5({\tt 1})\varphi^4({\tt 0})\varphi^2({\tt 0}){\tt 0} \quad \dots 
\end{array}
\end{equation}

The only BS factor of length two is ${\tt 10}$, it has left extensions ${\tt 0,2}$ and right extensions ${\tt 1,2}$. Applying Proposition~\ref{prop:BS} Item~\ref{prop:BSD}, we obtain that $\varphi({\tt 1})\varphi({\tt 0})$ is BS with left extensions ${\tt 0,1}$ and right extensions ${\tt 0,2}$.
Proposition~\ref{prop:BS} Item~\ref{prop:BSC} gives us that ${\tt 1}\varphi^2({\tt 1})\varphi^2({\tt 0}) {\tt 0}$ is BS with left extensions ${\tt 0,2}$ and right extensions ${\tt 1,2}$. This process can be iterated providing us again with infinitely many BS factors: 
\begin{equation}\label{eq:branch2}
\begin{array}{l}
{\tt 10} \to \varphi({\tt 1})\varphi({\tt 0}) \rightarrow {\tt 1}\varphi^2({\tt 1})\varphi^2({\tt 0}) {\tt 0} \to \varphi({\tt 1})\varphi^3({\tt 1})\varphi^3({\tt 0})\varphi({\tt 0}) \to \\
\to {\tt 1}\varphi^2({\tt 1})\varphi^4({\tt 1})\varphi^4({\tt 0})\varphi^2({\tt 0}){\tt 0} \to \varphi({\tt 1})\varphi^3({\tt 1})\varphi^5({\tt 1})\varphi^5({\tt 0})\varphi^3({\tt 0})\varphi({\tt 0}) \quad \dots
\end{array}
\end{equation}

Each BS factor $v$ of length greater than two has at least two synchronization points and the corresponding BS factor $w$ from Proposition~\ref{prop:BS} is non-empty. In other words, the BS factor $v$ makes part of one of the sequences~\eqref{eq:branch1} and~\eqref{eq:branch2} of BS factors.

As a consequence of Proposition~\ref{prop:BS} and the above arguments, we get a~complete description of BS factors in $\bf p$.
\begin{corollary}\label{coro:BS}
    Let $w$ be a non-empty BS factor in $\bf p$. Then it has one of the following forms:
    \begin{itemize}
        \item[$(\mathcal{A})$]\label{BS:A} $$w_A^{(n)} = {\tt 1}\varphi^2({\tt 1})\varphi^4({\tt 1})\cdots \varphi^{2n}({\tt 1})\varphi^{2n-1}({\tt 0})\varphi^{2n-3}({\tt 0})\cdots \varphi({\tt 0})$$ for $n \geq 1$. 
        If $n = 0$, then we set $w_A^{(0)} = {\tt 1}$.

        The Parikh vector of $w_A^{(n)}$ is the same as of the word ${\tt 1}\varphi({\tt 012})\varphi^3({\tt 012})\dots \varphi^{2n-1}({\tt 012})$.
        
        \item[$(\mathcal{B})$]\label{BS:B} $$w_B^{(n)} = \varphi({\tt 1})\varphi^3({\tt 1})\cdots \varphi^{2n+1}({\tt 1})\varphi^{2n}({\tt 0})\varphi^{2n-2}({\tt 0})\cdots \varphi^2({\tt 0}){\tt 0}$$ for $n \geq 0$. 

        The Parikh vector of $w_B^{(n)}$ is the same as of the word ${\tt 012}\varphi^2({\tt 012})\varphi^4({\tt 012})\dots \varphi^{2n}({\tt 012})$.

        \item[$(\mathcal{C})$]\label{BS:C} $$w_C^{(n)} = {\tt 1}\varphi^2({\tt 1})\varphi^4({\tt 1})\cdots \varphi^{2n}({\tt 1})\varphi^{2n}({\tt 0})\varphi^{2n-2}({\tt 0})\cdots \varphi^2({\tt 0}){\tt 0}$$ for $n \geq 0$.
        
        The Parikh vector of $w_C^{(n)}$ is the same as of the word ${\tt 01}\varphi^2({\tt 01})\varphi^4({\tt 01})\dots \varphi^{2n}({\tt 01})$.

        \item[$(\mathcal{D})$]\label{BS:D} $$w_D^{(n)} = \varphi({\tt 1})\varphi^3({\tt 1})\cdots \varphi^{2n+1}({\tt 1})\varphi^{2n+1}({\tt 0})\varphi^{2n-1}({\tt 0})\cdots \varphi({\tt 0})$$ for $n \geq 0$.        

        The Parikh vector of $w_D^{(n)}$ is the same as of the word $\varphi({\tt 01})\varphi^3({\tt 01})\dots \varphi^{2n+1}({\tt 01})$.
    \end{itemize}
\end{corollary}

\begin{lemma}\label{lem:ordinary}
    All BS factors in $\bf p$ are ordinary. 
\end{lemma}
\begin{proof}
The empty word is ordinary because all factors of length two are $\tt 10,01,02,12,21$. Thus $b(\varepsilon) = 5 - 3- 3 +1 = 0$.
It is easy to verify that each non-empty BS factor $w$ has three extensions. In particular,
\begin{itemize}
\item extensions of $w=w_A^{(n)}$ are: ${\tt 0}w{\tt 2}, {\tt 2}w{\tt 0}, {\tt 0}w{\tt 0}$,
\item extensions of $w=w_B^{(n)}$ are: ${\tt 2}w{\tt 2}, {\tt 2}w{\tt 1}, {\tt 0}w{\tt 2}$,
\item extensions of $w=w_C^{(n)}$ are: ${\tt 0}w{\tt 0}, {\tt 0}w{\tt 2}, {\tt 1}w{\tt 0}$,
\item extensions of $w=w_D^{(n)}$ are: ${\tt 1}w{\tt 2}, {\tt 0}w{\tt 1}, {\tt 1}w{\tt 1}$. 
\end{itemize}
Consequently, $b(w) = 3 - 2 - 2 + 1 = 0$.

\end{proof}

\subsubsection{The shortest return words to bispecial factors in $\bf p$}\label{sec:shortest_retwords}
Each factor of $\bf p$ has 3 return words. This claim follows from the next theorem.
\begin{theorem}[Theorem 5.7 in~\cite{BaPeSt2008}] Let $\bf{u}$ be a uniformly recurrent infinite word. Then each factor of $\bf{u}$ has exactly 3 return words if and only if $C(n) = 2n+1$ and $\bf{u}$ has no weak BS factors.
\end{theorem}

Let us first comment on return words to the shortest BS factors -- observe the prefix of $\bf p$ at the beginning of this section.
\begin{itemize}
    \item The return words to $\varepsilon$ are ${\tt 0,1,2}$. 
    \item The return words to ${\tt 1}$ are ${\tt 12,102,10}$.
    \item The return words to $\varphi({\tt 1}){\tt 0}$ are ${\tt 210} = \varphi({\tt 1}){\tt 0}$, ${\tt 21010}$, ${\tt 2101}$. The shortest one is ${\tt 210}$ and it is a prefix of all of them. 
    \item The return words to ${\tt 10}$ are ${\tt 10}$, ${\tt 102}$, ${\tt 1012}$. The shortest one is ${\tt 10}$ and it is a prefix of all of them. 
\end{itemize}

\begin{lemma}\label{lem:images_retwords}
If $w$ is a non-empty BS factor of $\bf p$ and $v$ is a return word to $w$, then $\varphi(v)$ is a return word to $\varphi(w)$.  
\end{lemma}
\begin{proof}
On one hand, since $vw$ contains $w$ as a prefix and as a suffix, $\varphi(v)\varphi(w)$ contains $\varphi(w)$ as a prefix and as a suffix, too. 
On the other hand, $w$ starts in ${\tt 1}$ or ${\tt 2}$ and ends in ${\tt 0}$ or ${\tt 1}$, thus $\varphi(w)$ starts in ${\tt 0}$ or ${\tt 2}$ and ends in ${\tt 1}$, therefore it has the following synchronization points $\bullet \varphi(w)\bullet$. Consequently,  $\varphi(v)\varphi(w)$ cannot contain $\varphi(w)$ somewhere in the middle because in such a case, by injectivity of $\varphi$, $vw$ would contain $w$ also somewhere in the middle. 
\end{proof}

The following observation is an immediate consequence of the definition of return words.
\begin{observation}\label{obs:retwordsLSandRS}
Let $w$ be a factor of $\bf p$ and let $v$ be its return word. 
If $w$ has a~unique right extension $a$, then $v$ is a~return word to $wa$, too. If $w$ has a~unique left extension $b$, then $bvb^{-1}$ is a return word to $bw$. In particular, the Parikh vectors of the corresponding return words are the same.  
\end{observation}

\begin{example}\label{ex} Consider the BS factor ${\tt 10}$ with the shortest return word ${\tt 10}$ (being a prefix of the other two return words), then by Lemma~\ref{lem:images_retwords} the BS factor $\varphi({\tt 1})\varphi({\tt 0})$ has the shortest return word equal to $\varphi({\tt 10})$. By Lemma~\ref{lem:images_retwords}, the factor $\varphi^2({\tt 1})\varphi^2({\tt 0})$ has the shortest return word equal to $\varphi^2({\tt 10})$ and by Observation~\ref{obs:retwordsLSandRS}, the shortest return word to the BS factor ${\tt 1}\varphi^2({\tt 1})\varphi^2({\tt 0}){\tt 0}$ has the same Parikh vector as $\varphi^2({\tt 10})$.  
\end{example}

Putting together Lemma~\ref{lem:images_retwords}, Observation~\ref{obs:retwordsLSandRS} and the knowledge of BS factors, we obtain the following statement about the shortest return words to BS factors in $\bf p$.
\begin{corollary}\label{coro:retwordsp} The shortest return words to BS factors in $\bf p$ have the following properties.

\begin{itemize}
        \item[$(\mathcal{A})$]\label{RW:A} The shortest return words to $w_A^{(n)}$ are
        \begin{itemize}
            \item[(i)] ${\tt 12}$ and ${\tt 10}$ for $n = 0$,
            \item[(ii)] $r_A^{(n)}$ with the same Parikh vector as $\varphi^{2n-1}({\tt 012})$ for $n\geq 1$.

        \end{itemize}

        \item[$(\mathcal{B})$]\label{RW:B} The shortest return word $r_B^{(n)}$ to $w_B^{(n)}$ has the same Parikh vector as $\varphi^{2n}({\tt 012})$. 

        \item[$(\mathcal{C})$]\label{RW:C}
            The shortest return word to $w_C^{(n)}$ is
            \begin{itemize}
                \item[(i)] ${\tt 10}$ for $n = 0$
                \item[(ii)] $r_C^{(n)}$ with the same Parikh vector as $\varphi^{2n}({\tt 01})$ for $n \geq 1$.

            \end{itemize}
            
        \item[$(\mathcal{D})$]\label{RW:D} 
        The shortest return word $r_D^{(n)}$ to $w_D^{(n)}$ has the same Parikh vector as $\varphi^{2n+1}({\tt 01})$.

    \end{itemize}
 \end{corollary} 
 \begin{proof}
We will prove case $(\mathcal{A})$. The other cases are similar. 
The shortest return words to $w_A^{(0)}=\tt 1$ are given at the beginning of Section~\ref{sec:shortest_retwords}. 
Let us proceed by induction on $n$. 
Consider the bispecial factor $w_A^{(1)} = {\tt 1}\varphi^2({\tt 1})\varphi({\tt 0})={\tt 1}\varphi(\varphi({\tt 1}){\tt 0})$. 
By description of the shortest return words to short bispecial factors, we know that $\tt 210$ is the shortest return word (moreover prefix of all other return words) to the bispecial factor $\varphi(\tt 1)\tt 0$. Using Lemma~\ref{lem:images_retwords}, $\varphi(\tt 210)$ is the shortest return word to the factor $\varphi^2(\tt 1)\varphi(\tt 0)$. By Observation~\ref{obs:retwordsLSandRS}, the Parikh vector of the shortest return word to $w_A^{(1)} = {\tt 1}\varphi^2({\tt 1})\varphi({\tt 0})$ is equal to the Parikh vector of $\varphi(\tt 210)$, hence also to the Parikh vector of $\varphi(\tt 012)$.
Assume for a fixed $n\geq 1$, the bispecial factor $w_A^{(n)} = {\tt 1}\varphi^2({\tt 1})\varphi^4({\tt 1})\cdots \varphi^{2n}({\tt 1})\varphi^{2n-1}({\tt 0})\varphi^{2n-3}({\tt 0})\cdots \varphi({\tt 0})$
has the shortest return word with the Parikh vector $\varphi^{2n-1}(\tt 012)$ and this return word is a prefix of all other return words.
By definition, $w_A^{(n+1)} = {\tt 1}\varphi^2(w_A^{(n)})\varphi({\tt 0})$. By Lemma~\ref{lem:images_retwords} and by induction assumption, the shortest return word to the factor $\varphi^2(w_A^{(n)})$ has the same Parikh vector as $\varphi^{2n+1}(\tt 012)$. Using Observation~\ref{obs:retwordsLSandRS}, we obtain that the shortest return word to $w_A^{(n+1)}$ has the same Parikh vector as the factor $\varphi^{2n+1}(\tt 012)$, too.
 \end{proof}

\subsubsection{The asymptotic critical exponent of $\bf p$}
Let us determine the asymptotic critical exponent of $\bf p$ using Theorem~\ref{thm:formulaE}.
We use the form of BS factors and their shortest return words determined above.
We get $E^*({\bf p})=1+\max\{A', B', C', D'\}$, where 
\begin{align*}
A' &= \limsup_{n\to\infty}\frac{|w_A^{(n)}|}{|{r}_A^{(n)}|} 
= \limsup_{n\to\infty} \frac{|{\tt 1}\varphi({\tt 012})\varphi^3({\tt 012})\dots \varphi^{2n-1}({\tt 012}))|}{|\varphi^{2n-1}({\tt 012})|}\,; \\
B' &= \limsup_{n\to\infty} \frac{|w_B^{(n)}|}{|{r}_B^{(n)}|} 
= \limsup_{n\to\infty} \frac{|{\tt 012}\varphi^2({\tt 012})\varphi^4({\tt 012})\dots \varphi^{2n}({\tt 012}))|}{|\varphi^{2n}({\tt 012})|}\,; \\
C' &= \limsup_{n\to\infty}\frac{|w_C^{(n)}|}{|{r}_C^{(n)}|}
= \limsup_{n\to\infty}\frac{|{\tt 01}\varphi^2({\tt 01})\varphi^4({\tt 01})\dots \varphi^{2n}({\tt 01})|}{|\varphi^{2n}({\tt 01})|}\,; \\
D' &= \limsup_{n\to\infty}\frac{|w_D^{(n)}|}{|{r}_D^{(n)}|}= \limsup_{n\to\infty}\frac{|\varphi({\tt 01})\varphi^3({\tt 01})\dots \varphi^{2n+1}({\tt 01})|}{|\varphi^{2n+1}({\tt 01})|}\,. 
\end{align*}
By the Hamilton-Cayley theorem, we have $M_\varphi^3-2M_\varphi^2+M_\varphi-I=0$. Consequently, for each $w \in \{{\tt 0,1,2}\}^*$, if we denote $\ell_n:=|\varphi^n(w)|=(1,1,1)M_{\varphi}^n \vec w$, then $\ell_n$ satisfies the recurrence relation $\ell_{n+3}-2\ell_{n+2}+\ell_{n+1}-\ell_n=0$.
Denote $\beta$ the largest root of the characteristic polynomial $t^3-2t^2+t-1$; $\beta \doteq 1.75488$.
By the Perron-Frobenius theorem, $\beta$ is strictly larger than the modulus of the other roots of the characteristic polynomial. We thus obtain:
$$\begin{array}{rcl}
A'&=&\lim\limits_{n \to \infty} \cfrac{\sum_{k=1}^n\beta^{2k-1}}{\beta^{2n-1}}=\frac{\beta^2}{\beta^2-1};\\
&&\\
B'=C'&=& \lim\limits_{n \to \infty} \cfrac{\sum_{k=0}^n\beta^{2k}}{\beta^{2n}}=\frac{\beta^2}{\beta^2-1};\\
&&\\
D'&=& \lim\limits_{n \to \infty} \cfrac{\sum_{k=0}^n\beta^{2k+1}}{\beta^{2n+1}}=\frac{\beta^2}{\beta^2-1}.\\
\end{array}$$
Consequently, $E^*({\bf p})=1+\frac{\beta^2}{\beta^2-1}\doteq 2.48$.

\subsection{The infinite word $\nu(\bf p)$}\label{sec:nu}
The morphism $\nu$ has the form:
\begin{align*}
    \nu({\tt 0}) &= {\tt 011},\\
    \nu({\tt 1}) &= {\tt 0},\\
    \nu({\tt 2}) &= {\tt 01}.  
\end{align*}
Therefore,
$$\nu({\bf p}) = {\tt 011001001101001100110100110010011} \cdots$$
and $\nu$ is injective.

\begin{remark}\label{rem:unique_preimage} The reader may easily check that any factor of $\nu(\bf p)$ of length at least two has a synchronization point.
\end{remark}
Using the above remark and Theorem~\ref{thm:E*_morphic_image}, we deduce that 
$$E^*(\nu({\bf p}))=E^*(\bf p).$$
\subsubsection{Bispecial factors in $\nu(\bf p)$}
\begin{lemma}\label{lem:nuBS}
    Let  $v \in \LL(\nu(\bf p))$ be a BS factor of length at least two.
    Then one of the items holds.
    \begin{enumerate}
    \item\label{theorem:BSAnu} There exists $w\in \LL(\bf p)$ such that $v = {\tt 1}\nu(w){\tt 01}$ and ${\tt 0}w,{\tt 2}w, w{\tt 0},w{\tt 2} \in \LL(\bf p)$.        
        \item\label{theorem:BSBnu} There exists $w\in \LL(\bf p)$ such that $v = \nu(w){\tt 0}$ and ${\tt 0}w,{\tt 1}w, w{\tt 1},w{\tt 2} \in \LL(\bf p)$.
        \item\label{theorem:BSCnu} There exists $w\in \LL(\bf p)$ such that $v = {\tt 1}\nu(w){\tt 0}$ and ${\tt 0}w,{\tt 2}w, w{\tt 1},w{\tt 2} \in \LL(\bf p)$.
            \item\label{theorem:BSDnu} There exists $w\in \LL(\bf p)$ such that $v = \nu(w){\tt 01}$ and ${\tt 0}w,{\tt 1}w, w{\tt 0},w{\tt 2} \in \LL(\bf p)$.  
        \end{enumerate}    
\end{lemma}
\begin{proof}
The statement follows from Remark~\ref{rem:unique_preimage} and from the possible left and right extensions of factors in $\bf p$.
\end{proof}

Combining Lemma~\ref{lem:nuBS} and Corollary~\ref{coro:BS}, we get a complete description of BS factors in $\nu(\bf p)$.
\begin{corollary}\label{coro:BS_nu}
    Let $v$ be a non-empty BS factor in $\nu(\bf p)$ of length at least two. Then $v= {\tt 01}$ or $v= {\tt 10}$ or $v$ has one of the following forms:
    \begin{itemize}
        \item[$(\mathcal{A})$]\label{BS:Anu} $$v_A^{(n)} = {\tt 1}\nu( {\tt 1}\varphi^2({\tt 1})\varphi^4({\tt 1})\cdots \varphi^{2n}({\tt 1})\varphi^{2n-1}({\tt 0})\varphi^{2n-3}({\tt 0})\cdots \varphi({\tt 0})){\tt 01}$$ for $n \geq 1$ and $v_A^{(0)}={\tt 1}\nu({\tt 1}){\tt 01}={\tt 1001}$. 
        
        The Parikh vector of $v_A^{(n)}$ is the same as of the word ${\tt 011}\nu({\tt 1}\varphi({\tt 012})\varphi^3({\tt 012})\dots \varphi^{2n-1}({\tt 012}))$.
        
        \item[$(\mathcal{B})$]\label{BS:Bnu} $$v_B^{(n)} = \nu(\varphi({\tt 1})\varphi^3({\tt 1})\cdots \varphi^{2n+1}({\tt 1})\varphi^{2n}({\tt 0})\varphi^{2n-2}({\tt 0})\cdots \varphi^2({\tt 0}){\tt 0}) {\tt 0}$$ for $n \geq 0$. 

        The Parikh vector of $v_B^{(n)}$ is the same as of the word ${\tt 0}\nu({\tt 012}\varphi^2({\tt 012})\varphi^4({\tt 012})\dots \varphi^{2n}({\tt 012}))$.

        \item[$(\mathcal{C})$]\label{BS:Cnu} $$v_C^{(n)} = {\tt 1}\nu({\tt 1}\varphi^2({\tt 1})\varphi^4({\tt 1})\cdots \varphi^{2n}({\tt 1})\varphi^{2n}({\tt 0})\varphi^{2n-2}({\tt 0})\cdots \varphi^2({\tt 0}){\tt 0}){\tt 0}$$ for $n \geq 0$.
        
        The Parikh vector of $v_C^{(n)}$ is the same as of the word ${\tt 01}\nu({\tt 01}\varphi^2({\tt 01})\varphi^4({\tt 01})\dots \varphi^{2n}({\tt 01}))$.

        \item[$(\mathcal{D})$]\label{BS:Dnu} $$v_D^{(n)} = \nu(\varphi({\tt 1})\varphi^3({\tt 1})\cdots \varphi^{2n+1}({\tt 1})\varphi^{2n+1}({\tt 0})\varphi^{2n-1}({\tt 0})\cdots \varphi({\tt 0})){\tt 01}$$ for $n \geq 0$.        

        The Parikh vector of $v_D^{(n)}$ is the same as of the word ${\tt 01}\nu(\varphi({\tt 01})\varphi^3({\tt 01})\dots \varphi^{2n+1}({\tt 01}))$.
    \end{itemize}
\end{corollary}

\subsubsection{The shortest return words to bispecial factors in $\nu(\bf p)$}

\begin{lemma}\label{lem:nu_retword}
If $w$ is a non-empty BS factor in $\bf p$ and $v$ is its return word, then $\nu(v)$ is a return word to $\nu(w){\tt 0}$. 
\end{lemma}
\begin{proof}
On one hand, consider any occurrence of $vw$ and denote $a$ the following letter, then $\nu(v)\nu(w){\tt 0}$ is a prefix of $\nu(vwa)$. Since $vw$ contains $w$ as a prefix and as a suffix, then $\nu(v)\nu(w){\tt 0}$ contains $\nu(w){\tt 0}$ as a prefix and as a suffix, too. 
On the other hand, $w$ starts in ${\tt 1}$ or ${\tt 2}$ and ends in ${\tt 0}$ or ${\tt 1}$, thus $\nu(w){\tt 0}$ starts in ${\tt 0}$ and ends in ${\tt 0110}$ or ${\tt 00}$, therefore $\nu(w){\tt 0}$ has the following synchronization points $\bullet \nu(w)\bullet {\tt 0}$. Consequently,  $\nu(v)\nu(w){\tt 0}$ cannot contain $\nu(w){\tt 0}$ somewhere in the middle because in such a case, by injectivity of $\nu$, $vw$ would contain $w$ also somewhere in the middle. 
\end{proof}
Applying Lemma~\ref{lem:nu_retword} and Observation~\ref{obs:retwordsLSandRS}, we have the following description of the shortest return words to BS factors.
\begin{corollary}\label{coro:retwords_nu}
The shortest return words to BS factors of length at least three in $\nu(\bf p)$ have the following properties.

\begin{itemize}
        \item[$(\mathcal{A})$]\label{RW:Anu} 
        The shortest return word $\hat{r}_A^{(n)}$ to $v_A^{(n)}$ has the same Parikh vector as $\nu(\varphi^{2n-1}({\tt 012}))$ for $n \geq 1$ and ${\tt 100}$ is the shortest return word to $v_A^{(0)}={\tt 1001}$.
        
        \item[$(\mathcal{B})$]\label{RW:Bnu}         
        The shortest return word $\hat{r}_B^{(n)}$ to $v_B^{(n)}$ has the same Parikh vector as $\nu(\varphi^{2n}({\tt 012}))$.

        \item[$(\mathcal{C})$]\label{RW:Cnu}
        The shortest return word $\hat{r}_C^{(n)}$ to $v_C^{(n)}$ has the same Parikh vector as $\nu(\varphi^{2n}({\tt 01}))$.

        \item[$(\mathcal{D})$]\label{RW:Dnu}
       The shortest return word $\hat{r}_D^{(n)}$ to $v_D^{(n)}$ has the same Parikh vector as $\nu(\varphi^{2n+1}({\tt 01}))$.
    \end{itemize}
\end{corollary}
\begin{proof}
We will prove case $(\mathcal{A})$. The other cases are similar. 
We know that the return words to $w_A^{(0)} = {\tt 1}$ in $\bf p$ are ${\tt 12}, {\tt 10}, { \tt 102}$.
Using Lemma~\ref{lem:nu_retword}, we obtain that $\nu({\tt 12})={\tt 001}$, $\nu({\tt 10}) ={\tt 0011}$ and $\nu({\tt 102}) = {\tt 001101}$ are return words to
$\nu({\tt 1})0$.
Since $\nu({\tt 1})0$ has unique left and right extensions $\tt 1$, using twice Observation~\ref{obs:retwordsLSandRS}, we obtain that ${\tt 1}\nu({\tt 12}){\tt 1}^{-1} = {\tt 100}$, $\tt 1001$ and $\tt 100110$ are return words to $v_A^{0} = {\tt 1001}$. 
Therefore, the shortest return word to $v_A^{(0)} = {\tt 1001}$ is ${\tt  100}$ and it is a prefix of all of them. 

Now, let us consider $n \geq 1$ and the bispecial factor 
$$v_A^{(n)} = {\tt 1}\nu( {\tt 1}\varphi^2({\tt 1})\varphi^4({\tt 1})\cdots \varphi^{2n}({\tt 1})\varphi^{2n-1}({\tt 0})\varphi^{2n-3}({\tt 0})\cdots \varphi({\tt 0})){\tt 01} = {\tt 1} \nu(w_A^{(n)}) {\tt 01 }.$$
Using Corollary~\ref{coro:retwordsp}, we know that the shortest return word to $w_A^{(n)}$ has the same Parikh vector as $\varphi^{2n-1}({\tt 012})$, moreover the shortest return word is a prefix of all other return words.

Using Lemma~\ref{lem:nu_retword}, and the fact that $\nu$ is non-erasing, we obtain that the shortest return word to $\nu(w_A^{(n)}){\tt 0 }$ has the same Parikh vector as $\nu(\varphi^{2n-1}({\tt 012}))$. 
Using Observation~\ref{obs:retwordsLSandRS} twice, we obtain that the shortest return word to  ${\tt 1}\nu(w_A^{(n)}){\tt 01}$ has the same Parikh vector as $\nu(\varphi^{2n-1}({\tt 012}))$, since adding $\tt 1$ at the beginning and erasing ${\tt 1}$ at the end does not change the Parikh vector. 
 \end{proof}

\subsubsection{The critical exponent of $\nu(\bf p)$}
Using Theorem~\ref{thm:formulaE} and the description of BS factors from Corollary~\ref{coro:BS_nu} and of their shortest return words from Corollary~\ref{coro:retwords_nu}, we obtain the following formula for the critical exponent of $\nu(\bf p)$.
$$
E(\nu({\bf p})) = 1 + \max\left\{A,B,C,D,F\right\}\,,
$$
where
\begin{align*}
A &= \sup\left\{\frac{|v_A^{(n)}|}{|\hat{r}_A^{(n)}|} : n\geq 1 \right\} 
= \sup\left\{\frac{|{\tt 011}\nu({\tt 1}\varphi({\tt 012})\varphi^3({\tt 012})\dots \varphi^{2n-1}({\tt 012}))|}{|\nu(\varphi^{2n-1}({\tt 012}))|} : n\geq 1 \right\} \cup \left\{\frac{|{\tt 1001}|}{|{\tt 100}|}\right\}\,; \\
B &= \sup\left\{\frac{|v_B^{(n)}|}{|\hat{r}_B^{(n)}|} : n\geq 0 \right\} 
= \sup\left\{\frac{|{\tt 0}\nu({\tt 012}\varphi^2({\tt 012})\varphi^4({\tt 012})\dots \varphi^{2n}({\tt 012}))|}{|\nu(\varphi^{2n}({\tt 012}))|} : n\geq 0 \right\}\,; \\
C &= \sup\left\{\frac{|v_C^{(n)}|}{|\hat{r}_C^{(n)}|} : n\geq 0 \right\} 
= \sup\left\{\frac{|{\tt 01}\nu({\tt 01}\varphi^2({\tt 01})\varphi^4({\tt 01})\dots \varphi^{2n}({\tt 01}))|}{|\nu(\varphi^{2n}({\tt 01}))|} : n\geq 1 \right\}\cup \left\{\frac{|{\tt 1}\nu({\tt 10}){\tt 0}|}{|\nu({\tt 10})|}\right\}\,; \\
D &= \sup\left\{\frac{|v_D^{(n)}|}{|\hat{r}_D^{(n)}|} : n\geq 0 \right\}
= \sup\left\{\frac{|{\tt 01}\nu(\varphi({\tt 01})\varphi^3({\tt 01})\dots \varphi^{2n+1}({\tt 01}))|}{|\nu(\varphi^{2n+1}({\tt 01}))|} : n\geq 0 \right\}\,;\\
F&=\max\left\{\frac{|w|}{|r|} : w \text{\ BS in $\nu(\bf p)$ of length one or two and $r$ its shortest return word} \right\}\,.
\end{align*}

\begin{theorem}
    The critical exponent of $\nu({\bf p})$ equals 
    \begin{equation*}
        E(\nu({\bf p})) = \frac{5}{2}.        
    \end{equation*}
\end{theorem}
\begin{proof}
To evaluate the critical exponent of $\nu(\bf p)$ using the above formula, we have to do several steps. 
\begin{enumerate}
\item Determining the shortest return words of BS factors of length one and two in $\nu(\bf p)$:
\begin{itemize}
    \item {\tt 0} is a BS factor with the shortest return word {\tt 0}. 
    \item {\tt 1} is a BS factor with the shortest return word {\tt 1}. 
    \item {\tt 01} is a BS factor with the shortest return words {\tt 010}, {\tt 011}. 
    \item {\tt 10} is a BS factor with the shortest return word {\tt 10}. 
    \end{itemize}
Thus for each BS factor $w$ of length one or two and its shortest return word $r$ we have
$\frac{|w|}{|r|}\leq 1$ and $F=1$.
\item Computation of $A$ and $B$.

Denote $c_n := |\nu(\varphi^n(012))|$, then it satisfies the recurrence relation $c_{n+1} = 2c_n - c_{n-1} + c_{n-2}$
with initial conditions $c_0 = 6, c_1 = 10$ and $c_2 = 17$.

The explicit solution reads
$$c_n = A_1\beta^n + B_1 \lambda_1^n + C_1 \lambda_2^n,$$
where
$$\beta \doteq 1.75488,\quad \lambda_1 \doteq 0.12256 + 0.74486 i, \quad \lambda_2 = \overline{\lambda_1}$$
are the roots of the polynomial $t^3 - 2t^2 + t - 1$, and
\begin{align*}
    A_1 &= \frac{6|\lambda_1|^2 - 20\operatorname{Re}(\lambda_1) + 17}{|\beta-\lambda_1|^2} \doteq 5.581308964\,;
\\
    B_1 &= \frac{6\beta\lambda_2 - 10 (\beta+\lambda_2)+17}{(\beta - \lambda_1)(\lambda_2-\lambda_1)} \doteq 0.209345518	-0.103481025	i\,;\\
    C_1 &= \overline{B_1}\,.
\end{align*}

Let us show that $A \leq \frac{3}{2}$. Since $\frac{|\tt 1001|}{|\tt 100|}=\frac{4}{3}<\frac{3}{2}$, it remains to show for all $n\geq 1$ that

\begin{align*}
    \frac{4 + A_1\sum_{k=1}^n\beta^{2k-1} + B_1 \sum_{k=1}^n \lambda_1^{2k-1} + C_1 \sum_{k=1}^n\lambda_2^{2k-1} }{A_1\beta^{2n-1} + B_1 \lambda_1^{2n-1} + C_1 \lambda_2^{2n-1}} 
    \quad &\leq^? \quad  \frac{3}{2},\\
    8 + 2A_1\sum_{k=1}^{n} \beta^{2k-1} + 4\operatorname{Re}\left(B_1 \sum_{k=1}^{n}  \lambda_1^{2k-1} \right) 
    \quad &\leq^? \quad  3A_1\beta^{2n-1} + 6\operatorname{Re}(B_1\lambda_1^{2n-1}),\\
8 + 2A_1\sum_{k=1}^{n-1} \beta^{2k-1} + 4\operatorname{Re}\left(B_1 \sum_{k=1}^{n-1}  \lambda_1^{2k-1} \right) 
    \quad &\leq^? \quad  A_1\beta^{2n-1} + 2\operatorname{Re}(B_1\lambda_1^{2n-1}),\\
    8 + 2A_1\left( \frac{\beta^{2n-1}}{\beta^2-1} - \frac{\beta}{\beta^2 - 1}\right) + 4\operatorname{Re}\left(B_1\lambda_1\frac{1-\lambda_1^{2n-2}}{1-\lambda_1^2} \right) 
    \quad &\leq^? \quad  A_1\beta^{2n-1} + 2\operatorname{Re}(B_1\lambda_1^{2n-1}).
\end{align*}

Since
\begin{align*}
\frac{2}{\beta^2-1} &\leq 1,
\end{align*}
we need to prove the inequality in the form
\begin{align*}
 8 + 4\operatorname{Re}\left(B_1\lambda_1\frac{1-\lambda_1^{2n-2}}{1-\lambda_1^2} \right) 
    \quad &\leq^? \quad  2A_1\frac{\beta}{\beta^2 - 1}+ 2\operatorname{Re}(B_1\lambda_1^{2n-1}).
\end{align*}

For the left side, we can write for $n \geq 1$
\begin{align*}
     8 + 4\operatorname{Re}\left(B_1\lambda_1\frac{1-\lambda_1^{2n-2}}{1-\lambda_1^2} \right) 
     \quad &\leq \quad 
     8 + 4 |B_1||\lambda_1|\frac{|\lambda_1|^{2n-2} + 1}{|\lambda_1^2-1|}\\
     \quad &\leq \quad 
     8 + 4 |B_1||\lambda_1|\frac{2}{|\lambda_1^2-1|}.
\end{align*}
For the right side, we can write for $n \geq 1$
\begin{align*}
    \quad  2A_1\frac{\beta}{\beta^2 - 1}+ 2\operatorname{Re}(B_1\lambda_1^{2n-1})
    \quad &\geq \quad
    2A_1\frac{\beta}{\beta^2 - 1} - 2 |B_1||\lambda_1|^{2n-1} \\
    &\geq \quad 2A_1\frac{\beta}{\beta^2 - 1} - 2 |B_1||\lambda_1|.
\end{align*}

Since the inequality 
\begin{equation*}
    8 + 4 |B_1||\lambda_1|\frac{2}{|\lambda_1^2-1|} \quad \leq\quad  2A_1\frac{\beta}{\beta^2 - 1} - 2 |B_1||\lambda_1|
\end{equation*}
holds true for the given values, we obtain $A\leq \frac{3}{2}$.

Next, we will show that $B \leq \frac{3}{2}$. 

For all $n \geq 0$ we have to show that
\begin{align*}
    \frac{1+ A_1\sum_{k=0}^n\beta^{2k} + B_1 \sum_{k=0}^n \lambda_1^{2k} + C_1 \sum_{k=0}^n\lambda_2^{2k} }{A_1\beta^{2n} + B_1 \lambda_1^{2n} + C_1 \lambda_2^{2n}} 
    \quad &\leq^? \quad  \frac{3}{2},\\
    2+2A_1\sum_{k=0}^{n-1}\beta^{2k} + 4\operatorname{Re}\left(B_1\sum_{k=0}^{n-1} \lambda_1^{2k}\right) 
    \quad &\leq^? \quad A_1\beta^{2n} + 2 \operatorname{Re}\left(B_1 \lambda_1^{2n}\right),\\
    2+2A_1\frac{\beta^{2n}-1}{\beta^2-1} + 4\operatorname{Re}\left(B_1\frac{\lambda_1^{2n}-1}{\lambda_1^2-1}\right) 
    \quad &\leq^? \quad A_1\beta^{2n} + 2 \operatorname{Re}\left(B_1 \lambda_1^{2n}\right).
    \end{align*}

Since
\begin{align*}
\frac{2}{\beta^2-1} &\leq 1,
\end{align*}
we need to prove the inequality in the form
\begin{align*}
    2 + 4\operatorname{Re}\left(B_1\frac{\lambda_1^{2n}-1}{\lambda_1^2-1}\right) 
    \quad &\leq^? \quad A_1\frac{2}{\beta^2-1}  + 2 \operatorname{Re}\left(B_1 \lambda_1^{2n}\right).
\end{align*}

For the left side, we can write for $n \geq 0$
\begin{align*}
     2 + 4\operatorname{Re}\left(B_1\frac{\lambda_1^{2n}-1}{\lambda_1^2-1}\right) 
    \quad &\leq \quad
    2 + 4|B_1| \frac{|\lambda_1|^{2n} +1}{|\lambda_1^2-1|}\\
    \quad &\leq \quad
    2 + 4|B_1| \frac{2}{|\lambda_1^2-1|}.    
\end{align*}

For the right side, we can write for $n \geq 0$
\begin{align*}
    A_1\frac{2}{\beta^2-1}  + 2 \operatorname{Re}\left(B_1 \lambda_1^{2n}\right)
    \quad &\geq \quad
    A_1\frac{2}{\beta^2 - 1} - 2 |B_1||\lambda_1|^{2n} \\
    &\geq \quad A_1\frac{2}{\beta^2 - 1} - 2 |B_1|.
\end{align*}

Since the inequality 
\begin{equation*}
    2 + 4|B_1| \frac{2}{|\lambda_1^2-1|}   
    \quad \leq\quad  
    A_1\frac{2}{\beta^2 - 1} - 2 |B_1|
\end{equation*}
holds true for given values, we obtain $B\leq \frac{3}{2}$.

\item Computation of $C$ and $D$.

Denote $d_n := |\nu(\varphi^n(01))|$, then it satisfies the same recurrence relation as $c_n$, i.e., $d_{n+1} = 2d_n - d_{n-1} + d_{n-2}$,
with initial conditions $d_0 = 4, d_1 = 7$ and $d_2 = 13$.

The explicit solution reads
$$d_n = A_2\beta^n + B_2 \lambda_1^n + C_2 \lambda_2^n,$$
where
$$\beta \doteq 1.75488,\quad \lambda_1 \doteq 0.12256 + 0.74486 i, \quad \lambda_2 = \overline{\lambda_1}$$
are the roots of the polynomial $t^3 - 2t^2 + t - 1$, and
\begin{align*}
    A_2 &= \frac{4|\lambda_1|^2 - 14\operatorname{Re}(\lambda_1) + 13}{|\beta-\lambda_1|^2} \doteq 4.213205567\,;
\\
    B_2 &= \frac{4\beta\lambda_2 - 7 (\beta+\lambda_2)+13}{(\beta - \lambda_1)(\lambda_2-\lambda_1)} \doteq -0.106602784	+	0.24671731	i\,;
\\
    C_2 &= \overline{B_2}\,.
\end{align*}

First, we will show that $C=\frac{3}{2}$. Recall that
$$C=\sup\left\{\frac{|01\nu(01\varphi^2(01)\varphi^4(01)\dots \varphi^{2n}(01))|}{|\nu(\varphi^{2n}(01))|} : n\geq 1 \right\}\cup \left\{\frac{|1\nu(10)0|}{|\nu(10)|}\right\}\,,$$ 
consequently, $C\geq \frac{|1\nu(10)0|}{|\nu(10)|}=\frac{3}{2}$.


It suffices to show for all $n \geq 1$ that
\begin{align*}
    \frac{2+ A_2\sum_{k=0}^n\beta^{2k} + B_2 \sum_{k=0}^n \lambda_1^{2k} + C_2 \sum_{k=0}^n\lambda_2^{2k} }{A_2\beta^{2n} + B_2 \lambda_1^{2n} + C_2 \lambda_2^{2n}} 
    \quad &\leq^? \quad  \frac{3}{2},\\
    4+2A_2\sum_{k=0}^{n-1}\beta^{2k} + 4\operatorname{Re}\left(B_2\sum_{k=0}^{n-1} \lambda_1^{2k}\right) 
    \quad &\leq^? \quad A_2\beta^{2n} + 2 \operatorname{Re}\left(B_2 \lambda_1^{2n}\right),\\
    4+2A_2\frac{\beta^{2n}-1}{\beta^2-1} + 4\operatorname{Re}\left(B_2\frac{\lambda_1^{2n}-1}{\lambda_1^2-1}\right) 
    \quad &\leq^? \quad A_2\beta^{2n} + 2 \operatorname{Re}\left(B_2 \lambda_1^{2n}\right).
\end{align*}

Since
\begin{align*}
\frac{2}{\beta^2-1} &\leq 1,
\end{align*}
we need to prove the inequality in the form
\begin{align*}
    4 + 4\operatorname{Re}\left(B_2\frac{\lambda_1^{2n}-1}{\lambda_1^2-1}\right) 
    \quad &\leq^? \quad A_2\frac{2}{\beta^2-1}  + 2 \operatorname{Re}\left(B_2 \lambda_1^{2n}\right).
\end{align*}

Now, we need to be more careful with the approximations. 
For the left side, we can write for $n \geq 2$
\begin{align*}
     4 + 4\operatorname{Re}\left(B_2\frac{\lambda_1^{2n}-1}{\lambda_1^2-1}\right) 
    \quad &\leq \quad     
     4 + 4\operatorname{Re}\left(\frac{B_2}{1-\lambda_1^2}\right) + 4|B_2|\frac{|\lambda_1|^{2n}}{|1-\lambda_1^2|}\\ 
    \quad &\leq \quad
    4 + 4\operatorname{Re}\left(\frac{B_2}{1-\lambda_1^2}\right) + 4|B_2|\frac{|\lambda_1|^{4}}{|1-\lambda_1^2|}.
\end{align*}

For the right side, we can write for $n \geq 2$
\begin{align*}
     A_2\frac{2}{\beta^2-1}  + 2 \operatorname{Re}\left(B_2 \lambda_1^{2n}\right)
     \quad &\geq \quad
    A_2\frac{2}{\beta^2 - 1} - 2 |B_2||\lambda_1|^{2n} \\
    &\geq \quad A_2\frac{2}{\beta^2 - 1} - 2 |B_2||\lambda_1|^{4}.
\end{align*}

Since the inequality 
\begin{equation*}
     4 + 4\operatorname{Re}\left(\frac{B_2}{1-\lambda_1^2}\right) + 4|B_2|\frac{|\lambda_1|^{4}}{|1-\lambda_1^2|}
    \quad \leq\quad  
    A_2\frac{2}{\beta^2 - 1} - 2 |B_2||\lambda_1|^{4}
\end{equation*}
holds true for given values, it remains to check the case for $n = 1$.

 If $n=1$, we get 
 \begin{equation*}
     \frac{2 + d_0 + d_2}{d_2} = \frac{19}{13} < \frac{3}{2}. 
 \end{equation*}

Therefore, we have proven that $C = \frac{3}{2}$.


It remains to prove $D \leq \frac{3}{2}$, however, the steps are the same as in the proof of the inequality $A \leq \frac{3}{2}$. Thus, we dare to skip it.



We have shown that $\max\{A, B, C, D\} = \frac{3}{2}$, and $F = 1$.
Consequently, $E(\nu({\bf p}))=1+\max\{A,B,C,D,F\}=\frac{5}{2}$.
\end{enumerate}
\end{proof}
\subsection{The infinite word $\mu(\bf p)$}\label{sec:mu}
The morphism $\mu$ has the form:
\begin{align*}
    \mu({\tt 0}) &= {\tt 011001},\\
    \mu({\tt 1}) &= {\tt 1001},\\
    \mu({\tt 2}) &= {\tt 0}.  
\end{align*}

{\scriptsize $$\mu({\bf p}) = {\tt 0110011001010010110010100101100110010110010100101100110010100101100110010110010} \dots$$}
and $\mu$ is injective.
\begin{remark}\label{rem:synchr_point_mu} The reader may easily check that any factor of $\mu(\bf p)$ of length at least six has a synchronization point.
\end{remark}
Using the above remark and Theorem~\ref{thm:E*_morphic_image}, we deduce that 
$$E^*(\mu({\bf p}))=E^*(\bf p).$$
\subsubsection{Bispecial factors in $\mu(\bf p)$}
\begin{lemma}\label{lem:muBS}
    Let  $v \in \LL(\mu(\bf p))$ be a BS factor of length at least six.
    Then one of the items holds.
    \begin{enumerate}
    \item\label{theorem:BSAmu} There exists $w\in \LL(\bf p)$ such that $v = \mu(w){\tt 01}$ and ${\tt 0}w,{\tt 2}w, w{\tt 0},w{\tt 2} \in \LL(\bf p)$.        
        \item\label{theorem:BSBmu} There exists $w\in \LL(\bf p)$ such that $v = {\tt 011001}\mu(w)$ and ${\tt 0}w,{\tt 1}w, w{\tt 1},w{\tt 2} \in \LL(\bf p)$.
        \item\label{theorem:BSCmu} There exists $w\in \LL(\bf p)$ such that $v = \mu(w)$ and ${\tt 0}w,{\tt 2}w, w{\tt 1},w{\tt 2} \in \LL(\bf p)$.
            \item\label{theorem:BSDmu} There exists $w\in \LL(\bf p)$ such that $v = {\tt 011001}\mu(w){\tt 01}$ and ${\tt 0}w,{\tt 1}w, w{\tt 0},w{\tt 2} \in \LL(\bf p)$.  
        \end{enumerate}    
\end{lemma}
We would like to point out that in this section, we use the same notation for BS factors and their shortest return words as in Section~\ref{sec:nu}. We are persuaded that no confusion arises since we do not refer here to the BS factors and their shortest return words from Section~\ref{sec:nu}.

\begin{corollary}\label{coro:BS_mu}    
Let $v$ be a BS factor in $\mu(\bf p)$ of length at least six. Then $v={\tt 011001}$ or $v={\tt 100101}$ or $v={\tt 01100101}$ or $v$ has one of the following forms:
    \begin{itemize}
        \item[$(\mathcal{A})$]\label{BS:Amu} 
        $$v_A^{(n)} = \mu({\tt 1}\varphi^2({\tt 1})\varphi^4({\tt 1})\cdots \varphi^{2n}({\tt 1})\varphi^{2n-1}({\tt 0})\varphi^{2n-3}({\tt 0})\cdots \varphi({\tt 0})){\tt 01}$$ for $n \geq 1$. 
        
        The Parikh vector of $v_A^{(n)}$ is the same as of the word ${\tt 01}\mu({\tt 1}\varphi({\tt 012})\varphi^3({\tt 012})\dots \varphi^{2n-1}({\tt 012}))$.
        
        \item[$(\mathcal{B})$]\label{BS:Bmu} 
        \begin{align*}
            v_B^{(n)} &= {\tt 011001}\mu(\varphi({\tt 1})\varphi^3({\tt 1})\cdots \varphi^{2n+1}({\tt 1})\varphi^{2n}({\tt 0})\varphi^{2n-2}({\tt 0})\cdots \varphi^2({\tt 0}){\tt 0})\\ 
            &= \mu({\tt 0}\varphi({\tt 1})\varphi^3({\tt 1})\cdots \varphi^{2n+1}({\tt 1})\varphi^{2n}({\tt 0})\varphi^{2n-2}({\tt 0})\cdots \varphi^2({\tt 0}){\tt 0})
        \end{align*} for $n \geq 0$. 

        The Parikh vector of $v_B^{(n)}$ is the same as of the word ${\tt 000111}\mu({\tt 012}\varphi^2({\tt 012})\varphi^4({\tt 012})\dots \varphi^{2n}({\tt 012}))$.

        \item[$(\mathcal{C})$]\label{BS:Cmu} $$v_C^{(n)} = \mu({\tt 1}\varphi^2({\tt 1})\varphi^4({\tt 1})\cdots \varphi^{2n}({\tt 1})\varphi^{2n}({\tt 0})\varphi^{2n-2}({\tt 0})\cdots \varphi^2({\tt 0}){\tt 0})$$ for $n \geq 0$.
        
        The Parikh vector of $v_C^{(n)}$ is the same as of the word $\mu({\tt 01}\varphi^2({\tt 01})\varphi^4({\tt 01})\dots \varphi^{2n}({\tt 01}))$.

        \item[$(\mathcal{D})$]\label{BS:Dmu} $$v_D^{(n)} = {\tt 011001}\mu(\varphi({\tt 1})\varphi^3({\tt 1})\cdots \varphi^{2n+1}({\tt 1})\varphi^{2n+1}({\tt 0})\varphi^{2n-1}({\tt 0})\cdots \varphi({\tt 0})){\tt 01}$$ for $n \geq 0$.        

        The Parikh vector of $v_D^{(n)}$ is the same as of the word ${\tt 00001111}\mu(\varphi({\tt 01})\varphi^3({\tt 01})\dots \varphi^{2n+1}({\tt 01}))$.
    \end{itemize}
\end{corollary}

\subsubsection{The shortest return words to bispecial factors in $\mu(\bf p)$}

\begin{lemma}\label{lem:mu_retword}
If $w$ is a BS factor of $\bf p$, $|w|\geq 2$, and $v$ is its return word, then $\mu(v)$ is a return word to $\mu(w)$. 
\end{lemma}
\begin{proof}
    On one hand, since $vw$ contains $w$ as a prefix and as a suffix, then $\mu(v)\mu(w)$ contains $\mu(w)$ as a prefix and as a suffix, too. 
On the other hand, $w$ starts in ${\tt 10}$ or ${\tt 21}$ and ends in ${\tt 0}$ or ${\tt 1}$, therefore $\mu(w)$ has the following synchronization points $\bullet \mu(w)\bullet $. Consequently,  $\mu(v)\mu(w)$ cannot contain $\mu(w)$ somewhere in the middle because in such a case, by injectivity of $\mu$, $vw$ would contain $w$ also somewhere in the middle. 

\end{proof}
Applying Lemma~\ref{lem:mu_retword} and Observation~\ref{obs:retwordsLSandRS}, we have the following description of the shortest return words to BS factors.

\begin{corollary}\label{coro:retwords_mu}
The shortest return words to BS factors of length greater than eight in $\mu(\bf p)$ have the following properties.

\begin{itemize}
        \item[$(\mathcal{A})$]\label{RW:Amu} 
        The shortest return word $\hat{r}_A^{(n)}$ to $v_A^{(n)}$ has the same Parikh vector as 
        $\mu(\varphi^{2n-1}({\tt 012}))$ for $n \geq 1$.
        
        \item[$(\mathcal{B})$]\label{RW:Bmu}         
        The shortest return word $\hat{r}_B^{(n)}$ to $v_B^{(n)}$ has the same Parikh vector as $\mu(\varphi^{2n}({\tt 012}))$.

        \item[$(\mathcal{C})$]\label{RW:Cmu}
        The shortest return word $\hat{r}_C^{(n)}$ to $v_C^{(n)}$ has the same Parikh vector as $\mu(\varphi^{2n}({\tt 01}))$.

        \item[$(\mathcal{D})$]\label{RW:Dmu}
       The shortest return word $\hat{r}_D^{(n)}$ to $v_D^{(n)}$ has the same Parikh vector as $\mu(\varphi^{2n+1}({\tt 01}))$.
    \end{itemize}
\end{corollary}
\begin{proof}
We will prove case $(\mathcal{A})$. The other cases are similar. 
Let us consider $n \geq 1$ and the bispecial factor 
$$v_A^{(n)} = \mu( {\tt 1}\varphi^2({\tt 1})\varphi^4({\tt 1})\cdots \varphi^{2n}({\tt 1})\varphi^{2n-1}({\tt 0})\varphi^{2n-3}({\tt 0})\cdots \varphi({\tt 0})){\tt 01} = \mu(w_A^{(n)}) {\tt 01 }.$$
Using Corollary~\ref{coro:retwordsp}, we know that the shortest return word to $w_A^{(n)}$ has the same Parikh vector as $\varphi^{2n-1}({\tt 012})$, moreover the shortest return word is a prefix of all of the return words.

Using Lemma~\ref{lem:mu_retword}, and the fact that $\mu$ is non-erasing, we obtain that the shortest return word to $\mu(w_A^{(n)})$ has the same Parikh vector as $\mu(\varphi^{2n-1}({\tt 012}))$. 
Using Observation~\ref{obs:retwordsLSandRS} twice, we obtain that the shortest return word to  $\mu(w_A^{(n)}){\tt 01}$ has the same Parikh vector as $\mu(\varphi^{2n-1}({\tt 012}))$. 
 \end{proof}

\subsubsection{The critical exponent of $\mu(\bf p)$}
Using Theorem~\ref{thm:formulaE} and the description of BS factors from Corollary~\ref{coro:BS_mu} and of their shortest return words from Corollary~\ref{coro:retwords_mu}, we obtain the following formula for the critical exponent of $\mu(\bf p)$.
$$
E(\mu({\bf p})) = 1 + \max\left\{A,B,C,D,F\right\}\,,
$$
where
\begin{align*}
A &= \sup\left\{\frac{|v_A^{(n)}|}{|\hat{r}_A^{(n)}|} : n\geq 1 \right\} 
= \sup\left\{\frac{|{\tt 01}\mu({\tt 1}\varphi({\tt 012})\varphi^3({\tt 012})\dots \varphi^{2n-1}({\tt 012}))|}{|\mu(\varphi^{2n-1}({\tt 012}))|} : n\geq 1 \right\}\,; \\
B &= \sup\left\{\frac{|v_B^{(n)}|}{|\hat{r}_B^{(n)}|} : n\geq 0 \right\} 
= \sup\left\{\frac{|{\tt 000111}\mu({\tt 012}\varphi^2({\tt 012})\varphi^4({\tt 012})\dots \varphi^{2n}({\tt 012}))|}{|\mu(\varphi^{2n}({\tt 012}))|} : n\geq 0 \right\}\,; \\
C &= \sup\left\{\frac{|v_C^{(n)}|}{|\hat{r}_C^{(n)}|} : n\geq 0 \right\} 
= \sup\left\{\frac{|\mu({\tt 01}\varphi^2({\tt 01})\varphi^4({\tt 01})\dots \varphi^{2n}({\tt 01}))|}{|\mu(\varphi^{2n}({\tt 01}))|} : n\geq 1 \right\}\,; \\
D &= \sup\left\{\frac{|v_D^{(n)}|}{|\hat{r}_D^{(n)}|} : n\geq 0 \right\}
= \sup\left\{\frac{|{\tt 00001111}\mu(\varphi({\tt 01})\varphi^3({\tt 01})\dots \varphi^{2n+1}({\tt 01}))|}{|\mu(\varphi^{2n+1}({\tt 01}))|} : n\geq 0 \right\}\,;\\
F&=\max\left\{\frac{|w|}{|r|} : w \text{\ BS in $\mu(\bf p)$ of length at most 8 and $r$ its shortest return word} \right\}\,.
\end{align*}

\begin{theorem}
    The critical exponent of $\mu({\bf p})$ equals 
    \begin{equation*}
        E(\mu({\bf p})) = \frac{28}{11}.        
    \end{equation*}
\end{theorem}
\begin{proof}
To evaluate the critical exponent of $\mu(\bf p)$ using the above formula, we have to do several steps. 
\begin{enumerate}
\item Determining the shortest return words of BS factors of length at most 8 in $\mu(\bf p)$:
\begin{itemize}
  \item {\tt 0} is a BS factor with the shortest return word {\tt 0}. 
  \item {\tt 1} is a BS factor with the shortest return word {\tt 1}.
  \item {\tt 01} is a BS factor with the shortest return word {\tt 01}.
  \item {\tt 10} is a BS factor with the shortest return word {\tt 10}.
  \item {\tt 010} is a BS factor with the shortest return word {\tt 01}.
  \item {\tt 1001} is a BS factor with the shortest return word {\tt 1001}.
  \item {\tt 011001} is a BS factor with the shortest return word {\tt 0110}.
  \item {\tt 100101} is a BS factor with the shortest return word {\tt 10010}.
  \item {\tt 01100101} is a BS factor with the shortest return word {\tt 011001}.
\end{itemize}
Therefore, $F = \max\left\{1, \frac{3}{2}, \frac{6}{5}, \frac{8}{6} \right\} < \frac{17}{11}$.
    
\item Computation of $A$ and $B$.

Denote $e_n := |\mu(\varphi^n(012))|$, then it satisfies the recurrence relation $e_{n+1} = 2e_n - e_{n-1} + e_{n-2}$
with initial conditions \\
$e_0 = 11, e_1 = 21$ and $e_2 = 36$.\\

The explicit solution reads
$$e_n = A_3\beta^n + B_3 \lambda_1^n + C_3 \lambda_2^n,$$
where
$$\beta \doteq 1.75488,\quad \lambda_1 \doteq 0.12256 + 0.74486 i, \quad \lambda_2 = \overline{\lambda_1}$$
are the roots of the polynomial $t^3 - 2t^2 + t - 1$, and
\begin{align*}
    A_3 &= \frac{11|\lambda_1|^2 - 42\operatorname{Re}(\lambda_1) + 36}{|\beta-\lambda_1|^2} \doteq 11.530751580 \,;
\\
    B_3 &= \frac{11\beta\lambda_2 - 21 (\beta+\lambda_2)+36}{(\beta - \lambda_1)(\lambda_2-\lambda_1)} \doteq 	-0.265375790 - 0.557144391 i \,;\\
    C_3 &= \overline{B_3}\,.
\end{align*}

Let us show that $A \leq \frac{17}{11}$. We have to show for all $n\geq 1$ that

\begin{align*}
    \frac{6 + A_3\sum_{k=1}^n\beta^{2k-1} + B_3 \sum_{k=1}^n \lambda_1^{2k-1} + C_3 \sum_{k=1}^n\lambda_2^{2k-1} }{A_3\beta^{2n-1} + B_3 \lambda_1^{2n-1} + C_3 \lambda_2^{2n-1}} 
    \quad &\leq^? \quad  \frac{17}{11},\\
    66 + 11A_3\sum_{k=1}^{n} \beta^{2k-1} + 22\operatorname{Re}\left(B_3 \sum_{k=1}^{n}  \lambda_1^{2k-1} \right) 
    \quad &\leq^? \quad  17A_3\beta^{2n-1} + 34\operatorname{Re}(B_2\lambda_1^{2n-1}),\\
66 + 11A_3\sum_{k=1}^{n-1} \beta^{2k-1} + 22\operatorname{Re}\left(B_3 \sum_{k=1}^{n-1}  \lambda_1^{2k-1} \right) 
    \quad &\leq^? \quad  6A_3\beta^{2n-1} + 12\operatorname{Re}(B_3\lambda_1^{2n-1}),\\
    66 + 11A_3\left( \frac{\beta^{2n-1}}{\beta^2-1} - \frac{\beta}{\beta^2 - 1}\right) + 22\operatorname{Re}\left(B_3\lambda_1\frac{1-\lambda_1^{2n-2}}{1-\lambda_1^2} \right) 
    \quad &\leq^? \quad  6A_3\beta^{2n-1} + 12\operatorname{Re}(B_3\lambda_1^{2n-1}).
\end{align*}

Since
\begin{align*}
\frac{11}{\beta^2-1} &\leq 6,
\end{align*}
we need to prove the inequality in the form
\begin{align*}
 66 + 22\operatorname{Re}\left(B_3\lambda_1\frac{1-\lambda_1^{2n-2}}{1-\lambda_1^2} \right) 
    \quad &\leq^? \quad  11A_3\frac{\beta}{\beta^2 - 1} + 12\operatorname{Re}(B_3\lambda_1^{2n-1}).
\end{align*}

For the left side, we can write for $n \geq 1$
\begin{align*}
      66 + 22\operatorname{Re}\left(B_3\lambda_1\frac{1-\lambda_1^{2n-2}}{1-\lambda_1^2} \right) 
     \quad &\leq \quad 
     66 + 22 |B_3||\lambda_1|\frac{|\lambda_1|^{2n-2} + 1}{|\lambda_1^2-1|}\\
     \quad &\leq \quad 
     66 + 22 |B_3||\lambda_1|\frac{2}{|\lambda_1^2-1|}.
\end{align*}
For the right side, we can write for $n \geq 1$
\begin{align*}
    \quad  11A_3\frac{\beta}{\beta^2 - 1}+ 12\operatorname{Re}(B_3\lambda_1^{2n-1})
    \quad &\geq \quad
    11A_3\frac{\beta}{\beta^2 - 1} - 12 |B_3||\lambda_1|^{2n-1} \\
    &\geq \quad 11A_3\frac{\beta}{\beta^2 - 1} - 12 |B_3||\lambda_1|.
\end{align*}

Since the inequality 
\begin{equation*}
    66 + 22 |B_3||\lambda_1|\frac{2}{|\lambda_1^2-1|} \quad \leq\quad  11A_3\frac{\beta}{\beta^2 - 1} - 12 |B_1||\lambda_1|
\end{equation*}
holds true for the given values, we obtain $A\leq \frac{17}{11}$.

Next, we will show that $B \leq \frac{17}{11}$. 

Since for $n = 0$ we have $\frac{|v_B^{(0)}|}{|\hat{r}_B^{(0)}|}=\frac{6+11}{11} = \frac{17}{11}$, it remains to show that 
for all $n \geq 1$ 
\begin{align*}
    \frac{6+ A_3\sum_{k=0}^n\beta^{2k} + B_3 \sum_{k=0}^n \lambda_1^{2k} + C_3 \sum_{k=0}^n\lambda_2^{2k} }{A_3\beta^{2n} + B_3 \lambda_1^{2n} + C_3 \lambda_2^{2n}} 
    \quad &\leq^? \quad  \frac{17}{11},\\
    66+11A_3\sum_{k=0}^{n-1}\beta^{2k} + 22\operatorname{Re}\left(B_3\sum_{k=0}^{n-1} \lambda_1^{2k}\right) 
    \quad &\leq^? \quad 6A_3\beta^{2n} + 12 \operatorname{Re}\left(B_3 \lambda_1^{2n}\right),\\
    66+11A_3\frac{\beta^{2n}-1}{\beta^2-1} + 22\operatorname{Re}\left(B_3\frac{\lambda_1^{2n}-1}{\lambda_1^2-1}\right) 
    \quad &\leq^? \quad 6A_3\beta^{2n} + 12 \operatorname{Re}\left(B_3 \lambda_1^{2n}\right).
    \end{align*}

Now, we need to be more careful with the approximations, we will therefore prove the inequality in the form
\begin{align*}
    66+ 22\operatorname{Re}\left(B_3\frac{\lambda_1^{2n}-1}{\lambda_1^2-1}\right) 
    \quad &\leq^? \quad 11A_3\frac{1}{\beta^2-1} + A_3\beta^{2n} \left(6-\frac{11}{\beta^2-1}\right) + 12 \operatorname{Re}\left(B_3 \lambda_1^{2n}\right).
\end{align*}

For the left side, we can write for $n \geq 1$
\begin{align*}
     66+ 22\operatorname{Re}\left(B_3\frac{\lambda_1^{2n}-1}{\lambda_1^2-1}\right) 
    \quad &\leq \quad
     66 + 22|B_3| \frac{|\lambda_1|^{2n} +1}{|\lambda_1^2-1|}\\
    \quad &\leq \quad
    66 + 22|B_3| \frac{1+|\lambda_1|^2}{|\lambda_1^2-1|}.    
\end{align*}

For the right side, we can write for $n \geq 1$
\begin{align*}
    A_3\frac{11}{\beta^2-1} + A_3\beta^{2n} \left(6-\frac{11}{\beta^2-1}\right) + 12 \operatorname{Re}\left(B_3 \lambda_1^{2n}\right)
    \quad &\geq \quad
    A_3\frac{11}{\beta^2 - 1} + A_3\beta^{2} \left(6-\frac{11}{\beta^2-1}\right)- 12 |B_3||\lambda_1|^{2n} \\
    &\geq \quad A_3\frac{11}{\beta^2 - 1}+ A_3\beta^{2} \left(6-\frac{11}{\beta^2-1}\right) - 12 |B_3||\lambda_1|^{2}.
\end{align*}

Since the inequality 
\begin{equation*}
   66 + 22|B_3| \frac{1+|\lambda_1|^2}{|\lambda_1^2-1|} 
    \quad \leq\quad  
A_3\frac{11}{\beta^2 - 1}+ A_3\beta^{2} \left(6-\frac{11}{\beta^2-1}\right) - 12 |B_3||\lambda_1|^{2}
\end{equation*}
holds true for the given values, we conclude $B = \frac{17}{11}$.

\item Computation of $C$ and $D$.
Denote $f_n := |\mu(\varphi^n(01))|$, then it satisfies the recurrence relation $f_{n+1} = 2f_n - f_{n-1} + f_{n-2}$
with initial conditions \\
$f_0 = 10, f_1 = 15$ and $f_2 = 26$.\\

The explicit solution reads
$$f_n = A_4\beta^n + B_4 \lambda_1^n + C_4 \lambda_2^n,$$
where
$$\beta \doteq 1.75488,\quad \lambda_1 \doteq 0.12256 + 0.74486 i, \quad \lambda_2 = \overline{\lambda_1}$$
are the roots of the polynomial $t^3 - 2t^2 + t - 1$, and
\begin{align*}
    A_4 &= \frac{10|\lambda_1|^2 - 30\operatorname{Re}(\lambda_1) + 26}{|\beta-\lambda_1|^2} \doteq 8.704306843 \,;
\\
    B_4 &= \frac{10\beta\lambda_2 - 15 (\beta+\lambda_2) + 26}{(\beta - \lambda_1)(\lambda_2-\lambda_1)} \doteq 	0.647846579 + 0.291191845 i \,;\\
    C_4 &= \overline{B_4}\,.
\end{align*}

The computation for $C \leq \frac{17}{11}$ is the same as for $B$.
Let us show that $D \leq \frac{17}{11}$. We have to show for $n\geq 1$ that

\begin{align*}
    \frac{8 + A_4\sum_{k=1}^n\beta^{2k-1} + B_4 \sum_{k=1}^n \lambda_1^{2k-1} + C_4 \sum_{k=1}^n\lambda_2^{2k-1} }{A_4\beta^{2n-1} + B_4 \lambda_1^{2n-1} + C_4 \lambda_2^{2n-1}} 
    \quad &\leq^? \quad  \frac{17}{11},\\
    88 + 11A_4\left( \frac{\beta^{2n-1}}{\beta^2-1} - \frac{\beta}{\beta^2 - 1}\right) + 22\operatorname{Re}\left(B_4\lambda_1\frac{1-\lambda_1^{2n-2}}{1-\lambda_1^2} \right) 
    \quad &\leq^? \quad  6A_4\beta^{2n-1} + 12\operatorname{Re}(B_4\lambda_1^{2n-1}).
\end{align*}

Since
\begin{align*}
\frac{11}{\beta^2-1} &\leq 6,
\end{align*}
we need to prove the inequality in the form
\begin{align*}
 88 + 22\operatorname{Re}\left(B_4\lambda_1\frac{1-\lambda_1^{2n-2}}{1-\lambda_1^2} \right) 
    \quad &\leq^? \quad  11A_4\frac{\beta}{\beta^2 - 1} + 12\operatorname{Re}(B_4\lambda_1^{2n-1}).
\end{align*}

For the left side, we can write for $n \geq 1$
\begin{align*}
      88 + 22\operatorname{Re}\left(B_4\lambda_1\frac{1-\lambda_1^{2n-2}}{1-\lambda_1^2} \right) 
     \quad &\leq \quad 
     88 + 22\operatorname{Re}\left(\frac{B_4\lambda_1}{1-\lambda_1^2} \right) + 22|B_4|\frac{|\lambda_1|^{2n-1}}{|1-\lambda_1^2|}\\
     \quad &\leq \quad 
     88 + 22\operatorname{Re}\left(\frac{B_4\lambda_1}{1-\lambda_1^2} \right) + 22|B_4|\frac{|\lambda_1|}{|1-\lambda_1^2|}.
\end{align*}
For the right side, we can write for $n \geq 1$
\begin{align*}
    \quad  11A_4\frac{\beta}{\beta^2 - 1}+ 12\operatorname{Re}(B_4\lambda_1^{2n-1})
    \quad &\geq \quad
    11A_4\frac{\beta}{\beta^2 - 1} - 12 |B_4||\lambda_1|^{2n-1} \\
    &\geq \quad 11A_4\frac{\beta}{\beta^2 - 1} - 12 |B_4||\lambda_1|.
\end{align*}

Since the inequality 
\begin{equation*}
   88 + 22\operatorname{Re}\left(\frac{B_4\lambda_1}{1-\lambda_1^2} \right) + 22|B_4|\frac{|\lambda_1|}{|1-\lambda_1^2|}
   \quad \leq\quad  
   11A_4\frac{\beta}{\beta^2 - 1} - 12 |B_4||\lambda_1|
\end{equation*}
holds true for the given values, we obtain $D\leq \frac{17}{11}$.

We have shown that $\max\{A, B, C, D\}=B=\frac{17}{11}$, and $F <\frac{17}{11} $.
Consequently, $E(\mu({\bf p}))=1+\max\{A,B,C,D,F\}=\frac{28}{11}$.
\end{enumerate}
\end{proof}


\end{document}